\documentclass[a4paper,10pt]{article}
\usepackage[utf8]{inputenc}
\usepackage{amsmath,amssymb,amsthm,color}
\usepackage{tikz}
\usepackage{hyperref}
\usepackage{MnSymbol}
\usepackage{accents}
\usepackage{enumerate}

\newtheorem{theorem}{Theorem}
\newtheorem{mainconstruction}[theorem]{Main Construction}
\newtheorem{proposition}[theorem]{Proposition}
\newtheorem{lemma}[theorem]{Lemma}
\newtheorem{corollary}[theorem]{Corollary}
\newtheorem{defin}[theorem]{Definition}
\newtheorem{example}[theorem]{Example}

\theoremstyle{remark}
\newtheorem{rem}[theorem]{Remark}

\usepackage[normalem]{ulem}
\renewcommand\sout{\bgroup\markoverwith%
{\textcolor{red}{\rule[0.7ex]{3pt}{1.4pt}}}\ULon}

\def\Mod{\mathop{\text{$\mathcal{M}$\kern-1.5pt\textit{od}}}\nolimits}
\def\Ric{\mathop{\mathrm{Ric}}\nolimits}

\def\End{\mathop{\mathrm{End}}\nolimits}

\def\Spin{\mathop{\mathrm{Spin}}\nolimits}
\def\spin{{\mathop{\mathfrak{spin}}}}
\def\SO{\mathop{\mathrm{SO}}\nolimits}
\def\so{{\mathop{\mathfrak{so}}}}

\def\U{\mathop{\mathrm{U}}\nolimits}
\def\SU{\mathop{\mathrm{SU}}\nolimits}
\def\Sp{\mathop{\mathrm{Sp}}\nolimits}
\def\GL{\mathop{\mathrm{GL}}\nolimits}

\def\divv{\mathop{\mathrm{div}}\nolimits}
\def\Diff{\mathop{\mathrm{Diff}}\nolimits}
\def\id{\mathop{\mathrm{id}}\nolimits}
\def\cl{\mathop{\mathrm{cl}}\nolimits}
\def\diag{\operatorname{diag}}    
\def\platz{\,\cdot\,}
\let\<\langle
\let\>\rangle
\def\Cl{\mathop{\mathrm{Cl}}\nolimits}
\def\CCl{\mathop{\mathbb{C}\mathrm{l}}\nolimits}
\def\pSigma{\Sigma^{\#}}
\def\ppSigma#1{\Sigma_{#1}^{(\#)}}
\let\Si\Sigma
\let\phi\varphi
\let\pa\partial
\newcommand{\univ}{\mathrm{univ}}

\let\ti\tilde   
\let\ol\overline

\let\ep\epsilon
\let\om\omega

\let\si\sigma
\let\theta\vartheta

\let\na\nabla

\DeclareUnicodeCharacter{00A0}{ }

\def\N{\mathbb{N}}
\def\Z{\mathbb{Z}}

\def\R{\mathbb{R}}
\def\C{\mathbb{C}}

\def\grad{{\rm grad}}

\def\cE{\mathcal{E}}
\def\cF{\mathcal{F}}
\def\cH{\mathcal{H}}
\def\cM{\mathcal{M}}
\def\cW{\mathcal{W}} 
\newcommand{\definedas}{\mathrel{\raise.095ex\hbox{\rm :}\mkern-5.2mu=}}
\newcommand\predash{(pre\nobreakdash-)}
\newcommand\premoduli{\predash mo\-du\-li}

\def\Joker{J^{(\#),(\pm)}}
\newcommand{\II}{\mathrm{I}\hskip-.3mm\mathrm{I}} 

\newcommand{\Or}{\mathcal{O}}

\newcommand{\bean}{\begin{eqnarray*}}
\newcommand{\eean}{\end{eqnarray*}}
\newcommand{\benu}{\begin{enumerate}}
\newcommand{\eenu}{\end{enumerate}}
\newcommand{\eea}{\end{eqnarray}}
\newcommand{\bea}{\begin{eqnarray}}

\renewcommand{\Re}{\mathop{\mathrm{Re}}}
\renewcommand{\Im}{\mathop{\mathrm{Im}}}

\let\mR\R
\let\mC\C
\let\mZ\Z
\let\mN\N
\let\ol\overline

\let\witi\widetilde
\newcounter{mnotecount}[section]

\newcommand{\arxiv}[1]{\href{https://arxiv.org/abs/#1}{arXiv:#1}}

\def\cF{\mathcal{F}}

\title{Construction of initial data sets for Lorentzian manifolds with lightlike parallel spinors}
\author{Bernd Ammann\thanks{B.A. has been partially supported by SFB 1085 Higher Invariants,
Regensburg, funded by the DFG.} 
\and Klaus Kr\"oncke \and Olaf Müller}

\begin{document}

\maketitle

\begin{abstract}
Lorentzian manifolds with parallel spinors are important objects of study in several branches of geometry, analysis and mathematical physics. Their Cauchy problem has recently been discussed by Baum, Leistner and Lischewski, who proved that the problem locally has a unique solution up to diffeomorphisms, 
provided that the intial data given on a space-like hypersurface satisfy some constraint equations. In this article we provide a method to solve these constraint equations. In particular, any curve (resp.\ closed curve) in the moduli space of Riemannian metrics on $M$ with a parallel spinor gives rise to a solution of the constraint equations on $M\times (a,b)$ (resp.\ $M\times S^1$).
\end{abstract}



\section{Introduction}

\subsection{Parallel spinors}

Let $(\overline N,\bar h)$ be a connected  $(n+1)$-dimensional oriented and time-oriented Lo\-ren\-tzian manifold with a fixed spin structure.
The bundle of complex spinors will be denoted by $\Sigma \overline N$. The spinor bundle carries a natural Hermitian form 
$\llangle\,\cdot\,,\,\cdot\,\rrangle$ of split signature, a compatible connection, and Clifford multiplication. In the present article we search for manifolds with a parallel spinor, i.e.\ a (non-trivial) parallel section of  $\Sigma \overline N$.
Understanding  Lorentzian manifolds with parallel spinors is interesting for several reasons. 

The first reason is that Riemannian manifolds with parallel spinors 
provide interesting structures. We want to briefly sketch some of them:
Parallel spinors provide a powerful technique to obtain Ricci-flat metrics on compact manifolds: All known closed Ricci-flat manifolds carry a non-vanishing parallel spinor on a finite covering.
Parallel spinors are also linked to ``stability'', defined in the sense that the given compact Ricci-flat metric cannot be deformed to a positive scalar curvature metric. This condition in turn is linked to dynamical stability of a Ricci-flat metric under Ricci flow: A compact Ricci-flat metric is dynamically stable under the Ricci flow if and only if it cannot be deformed to a positive scalar curvature metric (\cite{haslhofer.mueller.2014} and \cite[Theorem 1.1]{kroencke.2013}).

 Infinitesimal stability was proven for metrics with a parallel spinor in \cite{wang_m:91} and local stability (for manifolds with irreducible holonomy) in \cite{dai.wang.wei:05}.
Manifolds with parallel spinors provide interesting moduli spaces, see \cite{ammann.kroencke.weiss.witt:18}. Furthermore parallel spinors help to understand the space of metrics with non-negative scalar curvature. The stability property explained above implies that every homotopy class in the space of positive scalar curvature metrics which is known to be non-trivial also remains non-trivial
in the space of metrics with non-negative scalar curvature, see \cite{schick.wraith:2021}.  
It is interesting and challenging to see to which extent it is possible to find Lorentzian analogues to these results. 

A second reason to be interested in parallel spinors on arbitrary semi-Rie\-mann\-ian manifolds is that their existence implies that the holonomy is special \cite{leistner.diss,leistner.jdg2007}, \cite{berard-bergery.ikemakhen:93}, \cite{ikemakhen:04}, \cite{boubel:07}, \cite{baum.laerz.leistner:2014}. 
Thus the construction of Lorentzian manifolds with parallel spinors provides examples of manifolds with special holonomy.

A third reason is that parallel spinors are relevant in many fields of theoretical physics. For example parallel spinors on Lorentzian manifolds are often viewed as generators for the odd symmetries in a supersymmetric theory (see e.g. \cite{Farril.2000}). Parallel spinors are also important in several physical theories of Kaluza-Klein type, i.e.\ involving additional compactified dimensions. Mathematically we model them by a semi-Riemannian submersion $\pi:T\to B$ 
from a high-dimensional
Lorentzian manifold $T$ to a  macroscopically visible 3+1-dimensional space-time $B$. In this context one has to make sure that spectral and other analytic properties of Dirac operators on $T$ are comparable to the corresponding properties
on $B$. This requires the existence of harmonic spinors on the fibers $M:=\pi^{-1}(b)$, $b\in B$.
If the scalar curvature assumed to be non-negative, these harmonic spinors are parallel which in turn implies that the fibers are Ricci-flat.
Varying $b$ yields a family of metrics on $M$ with parallel spinors. One thus obtains a map $B\to \Mod_\parallel(M)$, where $\Mod_\parallel(M)$ is the \premoduli{} space of metrics with a parallel spinor on some covering of $M$, which is also a central ingredient in this article.

\subsection{The Cauchy problem for parallel light-like spinors}

Important progress about Lorentzian manifolds with parallel spinors was recently achieved by H.~Baum, T.~Leist\-ner and A.~Li\-schew\-ski \cite{baum.leistner.lischewski:16,lischewski:15-preprint,leistner.lischewski:17}, see also~\cite{baum.leistner.lecture.notes:HH} for associated lecture notes.
In particular, these authors showed the well-posedness of an associated
Cauchy problem which we will now describe in more detail and which will be the main topic of the present article.

Let again $(\ol N,\bar h)$ be a time- and space-oriented Lorentzian spin manifold, and let $\llangle\,\cdot\,,\,\cdot\,\rrangle$ be the Hermitian product on $\Sigma \ol N$ with split signature. The Clifford multiplication on  $(\ol N,\bar h)$ will be denoted by $\bullet$.

Note that for any spinor $\phi$
on a Lorentzian manifold one defines the \emph{Lorentzian Dirac current}
$V_\phi\in \Gamma(T\overline N)$ of $\phi$ 
 by requiring that 
  $$\bar h(X,V_\phi)=-\llangle X\bullet \phi,\phi\rrangle$$
holds for all $X\in T\overline N$. 
Recall that on Lorentzian manifolds 
the Clifford action of vector fields on spinors is symmetric, 
in contrast to the Riemannian case, where it is anti-symmetric.

If $\phi$ is parallel, then $V_\phi$ is 
parallel as well. One can show that~$V_\phi$ is a future oriented causal vector
\cite[Sec.~1.4.2, Prop.~2]{baum.leistner.lecture.notes:HH}. Thus $V_\phi$
is either time-like or light-like everywhere. In the time-like case, the Lorentzian manifold locally splits as a product $(N,h)\times (\mR,-dt^2)$, 
where $(N,h)$ is a Riemannian manifold with a parallel spinor. So with respect to a suitable Cauchy hypersurface the understanding of such metrics directly relies on the corresponding results on Riemannian manifolds. 

In this article we are concerned with the case, that $\phi$ is a parallel spinor with a light-like Dirac-current $V_\phi$. 
This problem was studied in \cite{baum.leistner.lischewski:16}, \cite{lischewski:15-preprint}, and \cite{leistner.lischewski:17}.

\begin{figure}
\begin{center}
\begin{tabular}{|c|c|c|l|c|c|c}
\hline
Manifold & Metric & Dimension & Type & Clifford& Scalar product\\
         &        &           &      & multiplication& on spinors\\
\hline
\hline
$\overline N$ &$\bar h$ & $n+1$ & Lorentzian &$\bullet$& $\llangle\platz,\platz\rrangle$\rule{0pt}{2.6ex}\\
$N$           & $h$ & $n$   & Riemannian &$\star$ & $\langle\platz,\platz\rangle$ \\
$M$           & $g$ & $m=n-1$ & Riemannian &$\cdot$& $\langle\platz,\platz\rangle$ \\
\hline
\end{tabular}\\[3mm]
\begin{tabular}{|c|c|}
\hline
Submanifold & Weingarten map \\
\hline
\hline
$N\subset \overline N$ &$\overline{W}$\rule{0pt}{2.6ex}\\
$M\subset N$           & $W$\\
\hline
\end{tabular}
\caption{Notation for manifolds, Clifford multiplications and Weingarten maps}
\end{center}
\end{figure}

Let $N$ be a spacelike hypersurface in this
Lorentzian manifold with induced metric $h$.
For a future-oriented (time-like) normal vector field $T$ with 
$\bar h(T,T)=-1$ we define the Weingarten map $\ol W:=-\bar\nabla T$. We use the symbol~$\ol W$ instead of~$W$, as the latter one 
will be used for the Weingarten map of hypersurfaces $M$ in $N$, 
which will play a central role in our article.
The Hermitian product $\llangle \,\cdot\,,\,\cdot\,\rrangle$ and the Clifford multiplication $\bullet$ on $\Sigma \ol N$ induce a positive definite Hermitian product and a Clifford multiplication $\star$ on $\Sigma \ol N|_N$ 
which can be characterized by the formulas
  $$\<\phi,\psi\>=\llangle T\bullet \phi,\psi\rrangle,\qquad X\star \phi= iT\bullet X \bullet \phi$$
for all $x\in N$, $\phi,\psi\in \Sigma_x \ol N$, $X\in T_xN$.
The spin structure on $\ol N$ induces a spin structure on the 
spacelike hypersurface $N$. Let $\Si N$ be the associated spinor bundle of $N$. 
With standard tools about Clifford modules, one sees 
that there is a bundle monomorphism $\iota:\Sigma N\to \Sigma \ol N|_N$ 
over the identity of $N$ such that the Clifford multiplication and the Hermitian product on $\Si N$ are mapped to $\star$ and $\<\,\cdot\,,\,\cdot\,\>$ on $\Si \ol N|_N$. The bundle monomorphism $\iota$ is a bundle isomorphism if and only if $n$ is even. However note that $\iota$ is not parallel, i.e.\ the connection is not preserved under~$\iota$. More precisely $\nabla \iota$ is a linear pointwise 
expression in $\ol W$. From now on we identify $\Sigma N$ with its image in  $\Si \ol N|_N$ under~$\iota$, taking the connection from~$\Sigma N$. As already mentioned above we will also consider hypersurfaces $M$ of $N$ and the spinor bundle $\Sigma M$ of $M$. However the relation between spinors on~$M$ and on~$N$ is a bit easier than between spinors on~$\ol N$ and on~$N$. One can work with an embedding $\Sigma M$ into $\Sigma N|_M$ which preserves the scalar product $\langle\platz,\platz\rangle$. We can even choose the embedding such that the Clifford multiplications coincide, however in the literature another embedding is often used which does no longer preserve the Clifford multiplication. In our application in Section~\ref{sec.constr.sol} it depends on the parity of $n$ which embedding is more convenient for us, see also Subsection~\ref{subsec.hyp.surf} and Appendix~\ref{sec.hypersurfaces} for further information about hypersurfaces and spinors. Thus we want to use $\cdot$ for the Clifford multiplication on $M$ in contrast to $\star$ which is used for the Clifford multiplication on $N$.
 
For any spinor $\phi\in\Gamma(\Sigma N)$ we associate --- see e.g.\ \cite[(1.7)]{leistner.lischewski:17} --- its Riemannian Dirac current $U_\phi\in \Gamma(TN)$ by requiring 
\begin{equation}\label{def.dirac.current}
h(U_\phi,X)=-i \<X\star \phi,\phi\>,\qquad \forall X\in TN.
\end{equation}
By Lemma~\ref{dirac.curr.prop.equiv} in Appendix~\ref{appendix.indep} we see, that the spinor $\phi$ 
satisfies $h(U_\phi,U_\phi) =\|\phi\|^4$ if and only if we have for any $p\in N$ 
\begin{equation}\label{char.quad.cond}
i\phi(p)\in \{V\star\phi\mid V\in T_pN\}.
\end{equation}
Now, if we assume that 
$(\ol N,\bar h)$ carries a non-vanishing parallel spinor, then this spinor induces a spinor $\phi$ on 
$(N,h)$ such that 
\begin{align}
  \nabla^N_X\phi& = \frac{i}2 \ol W(X)\star \phi,\qquad \forall X\in TN,\label{eq.imag.killing}\\
  U_\phi \star \phi & = i u_\phi\phi,\label{eq.Uphi}
\end{align}
where $U_\phi\in \Gamma(TN)$ is defined as above and where 
\begin{equation}\label{def.small.uphi}
 u_\phi:=\sqrt{h(U_\phi,U_\phi)},
\end{equation} 
see \cite[(4) and following]{baum.leistner.lischewski:16} or \cite[(1.6) and (1.8)]{leistner.lischewski:17}. 

Note that the constraint equations \eqref{def.dirac.current}--\eqref{def.small.uphi} are not independent from each others. We comment on this in Appendix~\ref{appendix.indep}.

Conversely, if $(N,h)$ is given as an abstract Riemannian spin manifold, and 
if $\ol W\in \Gamma(\End(TN))$, and $\phi\in \Gamma(\Sigma N)$ satisfy 
\eqref{char.quad.cond}, \eqref{eq.imag.killing} and \eqref{eq.Uphi} with~$U_\phi$ and~$u_\phi$ defined
by \eqref{def.dirac.current} and \eqref{def.small.uphi},
then there is a Lorentzian manifold $(\ol N,\bar h)$ with $N$ as a Cauchy hypersurface and a parallel spinor, such that $h$ is the induced Riemannian metric, $\ol W$ the Weingarten map and $\phi$ is induced by a parallel spinor on  $(\ol N,\bar h)$, see \cite[Consequence of Theorems~2 and~3]{leistner.lischewski:17}.
This was proven by solving the associated wave equations by using the technique of symmetric hyperbolic systems. A simpler approach, going back to a remark 
by P. Chrusciel was later given in \cite[Chap.~4]{seipel:19}.

The question arises on how to solve these constraint equations. In the present article 
we will describe a new method to obtain solutions of these constraint equations. 
We will see how any smooth curve in the \premoduli{} space of closed $m=(n-1)$-dimensional Riemannian manifolds with a parallel spinor together with a scaling function yields a solution to the constraint equations, see our Main Construction~\ref{mainconstrone} in Section~\ref{sec.curves.moduli}. And thus the well-definedness of the Cauchy problem implies the existence of an associated $(n+1)$-dimensional Lorentzian with a parallel spinor.  
Such a relation between families of metrics with special holonomy and solutions of the constraint equations was already conjectured by Leistner and Lischewski, see \cite{leistner.lischewski:17}. We essentially show that the conditions in 
\cite[Table 1]{leistner.lischewski:17} is satisfied if and only if the divergence condition \eqref{div.equation} in our Appendix~\ref{app.div.free} is satisfied.

We also derive versions of this construction to obtain initial data on a compact Cauchy hypersurface (without boundary). A first idea is to use a smooth closed
curve in the \premoduli{} space of closed $m$-dimensional Riemannian manifolds with a parallel spinor together with a scaling function. However this will in general not lead to a solution of the original constraint equation, but to a twisted version thereof, see Main Construction~\ref{mainconstrthree}.  This will lead to a Lorentzian manifold with a parallel twisted spinor. 
Here the twist bundle is always a complex line bundle with a flat connection.

In special cases -- more precisely assuming the ``fitting condition'' introduced in Section~\ref{sec.curves.moduli} we however obtain solutions of the constraint equations in the original (i.e. untwisted) sense, see Main Construction~\ref{mainconsttwo}.

We consider it as remarkable that compared to other classical diffeomorphism invariant Cauchy problems in Lorentzian geometry, e.g. the Cauchy problem for Ricci-flat metrics, we get a large quantity of solutions to the constraint equations. Furthermore it is amazing that the solutions in our situation correspond to curves in the moduli space, while the set of solutions of the constraints in classical problems have no similar description. An important input for our article was the smooth manifold structure of the premoduli space $\Mod_\parallel(M)$ and the fact that the BBGM connection preserves parallel spinors along divergence free Ricci-flat families of metrics. Ricci-flat deformations of a metric thus preserve the dimension of the space of parallel spinors. The infinitesimal version of this should be seen as some kind of Hodge theory: infinitesimal Ricci-flat deformations can be viewed as elements of $\ker (D^{T^*M})^2$ while infinitesimal deformations with parallel spinors can be viewed as elements of $\ker D^{T^*M}$ --- and on compact manifolds Hodge theory tells us that $\ker (D^{T^*M})^2=\ker D^{T^*M}$. 
Let us also compare the results of this paper to the tightly related recent preprint \cite{ammann.gloeckle:21p}. While the current paper constructs initial data for Lorentzian manifolds with a parallel spinor, one of the major themes in \cite{ammann.gloeckle:21p} is to get topological obstructions for a closed spin manifold to be a spacelike hypersurface in a Lorentzian manifold with a parallel spinor. We also mentioned that the special case of $(3+1)$-dimensional Lorentzian manifolds (i.e.\ initial data on $3$-dimensional manifolds) was studied in the recent preprint \cite{murcia.shabazi:20p}.

The structure of the article is as follows. In 
Section~\ref{sec.prelim} we fix some conventions, recall and extend known facts about Clifford modules and about spinors on hypersurfaces. We also explain how to differentiate a spinorial expression such as the Levi-Civita derivative $\na^g \phi$ of a spinor $\phi$ with respect to the Riemannian metric $g$.
The known results for defining this differential along a path of metrics is recalled in Subsection~\ref{subseq.bbgm}, this includes the BBGM parallel transport and the associated connection which arises from the universal spinor bundle construction; the concrete calculations are a central tool of the article and are 
carried out in Section~\ref{var.par.spinor.eq}. In Section~\ref{sec.par.preserved} we show that for divergence free Ricci-flat deformations, the BBGMconnection preserves parallel spinors, a result which we consider an interesting outcome of the article independently of the application to the constraint equations, mainly discussed in the article. In Section~\ref{sec.constr.sol} we use this construction to obtain solutions of the constraint equations \eqref{def.dirac.current} -- \eqref{def.small.uphi}. In Section~\ref{sec.curves.moduli} we establish the connection to curves in \premoduli{} spaces and we also discuss on how to obtain solutions on compact Cauchy hypersurfaces. The article ends with several appendices where we provide some details about facts which are already well-known, but where adequate literature was not available. We hope that these appendices help to make the article sufficiently self-contained.

{\bf Acknowledgements.} We thank Helga Baum, Thomas Leistner and Andree Lischewski for bringing our attention to this problem and for enlightening talks and discussions. 
Our special thank goes to Andree Lischewski for sharing the above mentioned conjecture with us, already at an early stage. We also thank both referees for many good and substantial comments.

\section{Preliminaries}\label{sec.prelim}

\subsection{Conventions}\label{conventions}

All Hermitian scalar products in this article are complex
linear in the first entry and complex anti-linear in the second one.
Let $E$ be a Riemannian vector bundle over a compact Riemannian manifold $(M,g)$ equipped with a metric connection $\nabla$ and $k\in \N_0$. Then the Sobolev norm $H^k$ of a section $s\in \Gamma(E)$ is
\begin{align*}
\left\|s\right\|_{H^k}=\left(\sum_{l=0}^k\int_M |\nabla^ls|^2dV\right)^{1/2}.
\end{align*}
where $dV$ is the volume element of $g$. As usual, we denote $L^2=H^0$. We will write $H^k(g)$ if we wish to emphasise the dependence of the norm on $g$.

We use the symbol $\bigodot^2$ for the symmetrized tensor product, i.e.\ 
for a finite-dimensional real vector space $V$ the space of symmetric bilinear
forms $V\times V\to\R$ is denoted by $\bigodot^2V^*$.
Let $S V^*\subset  \bigodot^2V^*$ be the subset of positive definite symmetric forms on $V$. Applying this fiberwise to the tangent bundle $TM$ we obtain the vector bundle $\bigodot^2 T^*M$ and the bundle $ST^*M=: S M$.
 
\begin{defin}
On a Riemannian manifold $M$ the Einstein operator is the elliptic 
differential operator 
$\Delta_E:\Gamma(\bigodot^2 T^*M) \to \Gamma(\bigodot^2 T^*M)$, given by
\begin{align*}
\Delta_Eh=\nabla^*\nabla h-2\mathring{R}h,\qquad \mathring{R}h(X,Y):=h(R_{e_i,X}Y,e_i)
\end{align*}
where $\left\{e_1,\ldots e_n\right\}$ is a local orthonormal frame.
Here, the curvature tensor is defined with the sign convention such that $R_{X,Y}Z=\nabla^2_{X,Y}Z-\nabla^2_{Y,X}Z$.
\end{defin}
The Einstein operator is linked to the deformation theory of Ricci-flat metrics as follows: Let $g$ be a Ricci-flat metric and $g_s$ a smooth family of Ricci-flat metrics through $g_0=g$, then
\begin{align*}
\frac{d}{ds}\Big|_{s=0}g_s=h+\mathcal{L}_Xg
\end{align*}
where $X$ is a vector field and $h\in \ker({\divv_g})\cap\ker(\Delta_E)$, i.e.\ the essential part of a Ricci-flat deformation is an element in $\ker(\Delta_E)$. In particular, if $g_s$ is a family of metrics with a parallel spinor, the essential part of its $s$-derivative is contained in $\ker(\Delta_E)$.
For more details on the deformation theory of Ricci-flat metrics, see \cite[Chapter 12 D]{besse:87}.

\subsection{Some facts about Clifford modules}\label{sec.Clifford.modules}

In this subsection we briefly summarize some facts about representations
of Clifford algebras. The interested reader might consult the first Chapter of the book by Lawson and Michelsohn \cite{lawson.michelsohn:89} for further details.

Let $(E_1,\ldots,E_n)$ be the canonical basis of $\R^n$. The complexified Clifford algebra for  $\R^n$ with the Euclidean scalar product will be denoted by $\CCl_n$. In this section all Clifford multiplications are written as $\cdot$, independently on $n$; we use the convention that $\cdot$ does not depend on $n$.
In the following section we will then explain that the associated bundle construction then turns this Clifford multiplication both into the Clifford multiplication $\cdot$ on $M$ and the Clifford multiplication $\star$ on $N$.

We define
the complex volume element $\omega^\C_n$ as
$$\omega^\C_n:= \epsilon_n E_1\cdot E_2 \cdot \ldots \cdot E_n\in \CCl_n,$$
where $\epsilon_{4k}=\epsilon_{4k+3}=1$, $\ep_{4k+1}=-i$ and $\ep_{4k+2}=i$.
The choices imply $(\omega^\C_n)^2=1$. The choice of sign for $\epsilon_n$ in the literature varies between different sources, our particular and unconventional 
choice yields an easy formulation of Lemma~\ref{lem.algeb.repr}.

For~$n$ even, there is a unique irreducible complex represention of $\CCl_n$. 
Here and in the following ``unique'' will always ``unique up to isomorphism of representation''. 
This unique representation will be denoted by $\Sigma_n$. 
It comes with a grading $\Sigma_n=\Sigma_n^+ \oplus \Sigma_n^-$ 
given by the eigenvalues $\pm 1$ of $\omega^\C_n$.

For~$n$ odd there are two irreducible representations, and as
$\omega^\C_n$ is central, Schur's Lemma implies that $\omega^\C_n$ acts either as the identity or minus the identity which allows us to distinguish
the two representations. 
For $n$ odd,
we will assume that $\Sigma_n$ is the irreducible representation on which 
$\omega^\C_n$ acts as the identity.
The other representation is called $\pSigma_n$. 
In the following $\ppSigma{n}$ denotes either  $\Sigma_n$ or $\pSigma_n$ 
for odd $n$ and  $\Sigma_n$ for even $n$.

The standard inclusion $\R^{n-1}\to \R^n$, $x\mapsto (x,0)$
induces an inclusion $\CCl_{n-1}\to \CCl_n$. This turns $\Sigma_n$
into a $\CCl_{n-1}$-module. In particular, the Clifford action of $X\in \R^{n-1}$
is the same if $X$ is viewed as an element of $\CCl_{n-1}$ or $\CCl_n$.
\footnote{Here we do not follow the convention chosen by Bär,
Gauduchon and Moroianu \cite{baer.gauduchon.moroianu:05}, where vectors
in $\R^{n-1}$ have two different Clifford actions on  $\Sigma_n$, depending on whether it is seen as an element in  $\CCl_{n-1}$ or $\CCl_n$-action.}

Thus, if $n$ is odd,  we can choose (and fix from now on) 
isometric isomorphisms of $\CCl_{n-1}$-modules $J_n:\Sigma_{n-1}\to \Sigma_n$
and $J_n^\#:\Sigma_{n-1}\to \pSigma_n$.
If $n$ is even, then we can choose (and fix from now on) 
an isometric isomorphism of $\CCl_{n-1}$-modules 
$I_n:\Sigma_{n-1}\oplus\pSigma_{n-1}\to\Sigma_n$.

\begin{lemma}\label{lem.algeb.repr}
For $n$ odd, $J_n:\Sigma_{n-1}\to \Sigma_n$ and $J_n^{\#}:\Sigma_{n-1}\to \pSigma_n$
are $\CCl_{n-1}$-linear isomorphisms of complex $\Spin(n-1)$-representations and they satisfy
\begin{eqnarray}
 e_n\cdot J_n(\phi) &=& i J_n(\omega_{n-1}^{\C}\phi),\label{J.id}\\
 e_n\cdot J_n^{\#}(\phi) &=& - i J_n^{\#}(\omega_{n-1}^{\C}\phi).\label{J.id.mod}
\end{eqnarray}
for all $\phi\in \Sigma_{n-1}$. 
For $n$ even there are isometric 
monomorphisms of complex $\Spin(n-1)$-representations
$J_n^{\pm}:\Sigma_{n-1}\to \Sigma_n$ and $J_n^{\#,\pm}:\pSigma_{n-1}\to \Sigma_n$, 
satisying
\begin{eqnarray}
 e_n\cdot J_n^{\pm}(\phi) &=& \pm i J_n^{\pm}(\phi),\label{J.id.even.0}\\
 e_n\cdot J_n^{\#,\pm}(\phi) &=& \pm i J_n^{\#,\pm}(\phi).\label{J.id.mod.even.0}
\end{eqnarray}
for all $\phi\in \Sigma_{n-1}^{(\#)}$  and for any $\pm\in\{+,-\}$.
\end{lemma}

\begin{proof}
If $n$ is odd, then we see that $\epsilon_n=-i\epsilon_{n-1}$, thus
$\omega_n^\C=-i\omega_{n-1}^\C\cdot e_n=-ie_n\cdot \omega_{n-1}^\C$. This implies 
$e_n\cdot \omega_n^\C=i\omega_{n-1}^\C$. Hence, we obtain \eqref{J.id}:
 $$e_n\cdot J_n(\phi)= e_n\cdot \omega_n^\C\cdot J_n(\phi)=i\omega_{n-1}^\C J_n(\phi)=J_n(i\omega_{n-1}^\C \phi).$$
The verification of \eqref{J.id.mod} is analogous.
As these morphisms are $\CCl_{n-1}$-linear, and as $\Spin(n-1)$ is contained in  $\CCl_{n-1}$, the morphisms are also $\Spin(n-1)$-equivariant.

If $n$ is even, then  we define
  $$J_n^{(\#),\pm}(\phi):= \frac{1}{\sqrt 2} \left(I_n(\phi)\mp i  e_n\cdot I_n(\phi)\right).$$
Equations~\eqref{J.id.even.0} and \eqref{J.id.mod.even.0} are obvious.
The morphisms are no longer $\CCl_{n-1}$-linear. However $e_n$ commutes with 
$(\CCl_{n-1})_0$ which is defined as the even part of $\CCl_{n-1}$, i.e.\ the sub-algebra generated by elements of the form $X\cdot Y$ with $X,Y\in \R^{n-1}$.
As a consequence the morphisms are $(\CCl_{n-1})_0$-linear.  
As $\Spin(n-1)$ is contained in  
$(\CCl_{n-1})_0$, the morphisms are $\Spin(n-1)$-equivariant.

To show that the morphisms $J_n^{(\#),\pm}$ are isometric 
(and thus also injective) 
it suffices to show $I_n(\phi) \perp i  e_n\cdot I_n(\phi)$, where $\phi\in \Sigma_{n-1}$ or  $\phi\in \pSigma_{n-1}$. We argue only for $\phi\in \Sigma_{n-1}$, 
as the other case is completely analogous. As $\omega_{n-1}^\C$ and $e_n$ 
anticommute, $i e_n\cdot I_n(\phi)$ is an eigenvector of $\omega_{n-1}^\C$
to the eigenvalue $-1$, while $I_n(\phi)$ is an eigenvector to the eigenvalue 
$1$. As $\omega_{n-1}^\C$ is self-adjoint, orthogonality follows.
\end{proof}

\begin{rem}
Let $n$ be even. 
Clifford multiplication with $e_n$ anticommutes with all odd elements 
in  $\CCl_{n-1}$, in particular with $\omega_{n-1}^{\C}$ and vectors 
$\R^{n-1}\subset \CCl_{n-1}$.
Thus, the map $\phi\mapsto I_n^{-1}(e_n\cdot I_n(\phi))$
restricts to a vector space isomorphism $A:\Sigma_{n-1}\to \pSigma_{n-1}$
which anticommutes with Clifford multiplication by vectors in~$\R^{n-1}$.
\end{rem}

\subsection{Spinors on hypersurfaces}\label{subsec.hyp.surf}

In this section we want to describe how one can restrict a
spinor on an $n$-dimensional Riemannian spin manifold $(N,h)$ to a spinor on an oriented hypersurface $M$ carrying the induced metric $g$. As, this
restriction is local, we can assume --- by restricting to a tubular neighborhood of $M$ and using Fermi coordinates, i.e.\ normal coordinates in normal directions --- that  $N= M\times (a,b)$ and  $h=g_s+ds^2$ where $s\in (a,b)$.

Conversely, given a family of metrics $g_s$ and spinors $\phi_s\in \Gamma(\Sigma^{g_s}M)$, $s\in (a,b)$, we want to obtain a spinor on $N:= M\times (a,b)$, $h=g_s+ds^2$.

Note that this description 
does not include the passage from a Lorentzian manifold to a spacelike hypersurface and vice versa, described in the introduction, which plays an important
role in the work of Baum, Leistner and Li\-schew\-ski. For this Lorentzian version 
analogous techniques are well presented in the lecture notes \cite[Sec.~1.5]{baum.leistner.lecture.notes:HH}.

For the Riemannian case, which is the goal of this section, a standard reference
is \cite{baer.gauduchon.moroianu:05}. However for the purpose of our article
it seems to be more efficient to choose some other convention, following
e.g.~\cite[Sec.~5.3]{ammann:habil}. To have a fluently readable summary, we do not include detailed proofs in this section. For self-containedness 
we include them in Appendix~\ref{sec.hypersurfaces}.

In the following let $T_MN$ be the vertical bundle of the projection $N\to (a,b)$, i.e.\
$T_MN=\coprod_{s\in (a,b)}T(M\times\{s\})$, viewed as a bundle over $N$. 

For each $s\in I:=(a,b)$ let $P_{\SO}(M,g_s)$ be the $\SO(n-1)$-principal bundle of positively oriented orthornormal frames on~$(M,g_s)$, and let $P_{\SO}(N)$
be defined as the corresponding $\SO(n)$-bundle over $(N,h)$.
We note $\nu:=\partial/\partial s$.
The union of the bundles  $P_{\SO}(M,g_s)$ is an $\SO(n-1)$-principal bundle over $N=M\times I$, which we will denoted by $P_{\SO,M}(N)\to N$ (whose topology and bundle structure is induced by the $\SO(n-1)$-reduction of the $\GL(n-1)$-principal bundle of all frames over $N$ containing $\nu$). We get a map $P_{\SO,M}(N)\to P_{\SO}(N)$, mapping $(e_1,\ldots,e_{n-1})$ to  $(e_1,\ldots,e_{n-1},\nu)$ which yields an isomorphism  $P_{\SO}(N)\cong P_{\SO,M}(N)\times_{\SO(n-1)}\SO(n)$.
 Any topological spin structure on $M$ yields a $\Spin(n-1)$-principal bundle $P_{\Spin,M}(N)$ with a  $\Spin(n-1)\to\SO(n-1)$-equivariant map  $P_{\Spin,M}(N)\to P_{\SO,M}(N)$ in the usual way.  This map induces   
  $$P_{\Spin}(N):= P_{\Spin,M}(N)\times_{\Spin(n-1)}\Spin(n)\to P_{\SO}(N)\cong P_{\SO,M}(N)\times_{\SO(n-1)}\SO(n)$$ 
which is a spin structure on $N=M\times I$. In particular this induces a bijection
from the set of spin structures on $M$ (up to isomorphism) to spin structures on~$N$ (up to isomorphism).

The complex volume elements $\om_{n-1}^\C\in \CCl_{n-1}$ and  $\om_n^\C\in \CCl_n$ then provide associated complex volume elements
  $$\om_M^\C =\epsilon_{n-1} e_1\cdot e_2 \cdot \ldots \cdot e_{n-1}\in \CCl(TM)\text{ and }\om_N^\C=\epsilon_n e_1\star e_2 \star \ldots \star e_n\in \CCl(TN),$$
where $(e_1,e_2,\ldots)$ is any positively oriented basis in $TM$ resp.\ $TN$.

Using the associated bundle construction, we obtain for $n$ odd:
\begin{align*}
  &\Sigma_M N:=  P_{\Spin,M}(N)\times_{\Spin(n-1)}\Sigma_{n-1}=\coprod_{s\in (a,b)} \Sigma^{g_s} M,\\ 
  &\Sigma N:=  P_{\Spin}(N)\times_{\Spin(n)}\Sigma_{n},\\
  &\pSigma N:=  P_{\Spin}(N)\times_{\Spin(n)}\pSigma_{n}.
\end{align*}
The grading $\Sigma_{n-1}=\Sigma_{n-1}^+ \oplus \Sigma_{n-1}^-$  yields a grading $\Sigma_M N=(\Sigma_M N)^+\oplus (\Sigma_M N)^-$ which is the grading by eigenvalues $\pm 1$ of $\omega_M^\C$.

For $n$ even we have
\begin{align*}
  &\Sigma_M N:=  P_{\Spin,M}(N)\times_{\Spin(n-1)}\Sigma_{n-1}=\coprod_{s\in (a,b)} \Sigma^{g_s} M,\\ 
  &\pSigma_M N:=  P_{\Spin,M}(N)\times_{\Spin(n-1)}\pSigma_{n-1}=\coprod_{s\in (a,b)} \Sigma^{\#,g_s} M,\\
  &\Sigma N:=  P_{\Spin}(N)\times_{\Spin(n)}\Sigma_{n},\\
\end{align*}
and the grading $\Sigma_n=\Sigma_n^+ \oplus \Sigma_n^-$ yields a grading $\Sigma N=(\Sigma N)^+ \oplus (\Sigma N)^-$, given by eigenvalues $\pm 1$ of $\omega_N^\C$.

The Clifford multiplications $\R^{n-1}\otimes \Sigma_{n-1}^{(\#)}\to  \Sigma_{n-1}^{(\#)}$,
$\R^n\otimes \Sigma_n^{(\#)}\to  \Sigma_n^{(\#)}$, induce Clifford multiplications
 $\cdot:T_MN\otimes \Sigma_M^{(\#)}N\to  \Sigma_M^{(\#)}N$, $\star:TN\otimes \Sigma^{(\#)} N\to  \Sigma^{(\#)} N$. As indicated before we use the symbols ``$\cdot$'' and ``$\star$'' to distinguish properly between the Clifford multiplication of $M$ and the one of $N$. 

Furthermore the maps $J_n$, $J_n^\#$, $J_n^{\pm}$, and $J_n^{\#,\pm}$  from Lemma~\ref{lem.algeb.repr} induce vector bundle maps over $\id:N\to N$ which are isometric injective $\C$-linear maps 
in each fiber. For $n$ odd, we obtain fiberwise isomorphisms
\begin{align*}
  &J:\Sigma_MN\to \Sigma N\text{ and }J^{\#}:\Sigma_MN\to \pSigma N.
\end{align*}
These maps commute with Clifford multiplication by vectors tangent to $M$.
They satisfy
\begin{eqnarray}
 \nu\star J(\phi) &=& i J_n(\omega_M^{\C}\cdot \phi),\label{J.id.N}\\
 \nu\star J^{\#}(\phi) &=& - i J^{\#}(\omega_M^{\C}\cdot\phi).\label{J.id.mod.N}
\end{eqnarray}
for all $\phi\in \Sigma_MN$.

On the other hand we get for $n$ even
\begin{align*}
  &J^\pm:\Sigma_M N\to \Sigma N\text{ and }J^{\#,\pm}:\pSigma_MN\to \Sigma N.
\end{align*}

These maps \emph{do not} commute with Clifford multiplication by vectors tangent to $M$, and they \emph{are not} surjective in any fiber.
However, they satisfy
\begin{eqnarray}
 \nu\star J^{\pm}(\phi) &=& \pm i J^{\pm}(\phi),\label{J.id.even}\\
 \nu\star J^{\#,\pm}(\phi) &=& \pm i J^{\#,\pm}(\phi).\label{J.id.mod.even}
\end{eqnarray}
for all $\phi\in \Sigma_MN$ and for any $\pm\in\{+,-\}$.


The bundle maps defined above do not preserve the (partially defined) 
Levi-Civita connections on the bundles.
They modify the connection by terms depending on the second fundamental form of $M\times \{s\}$ in $N$. This is made precise in the following lemma.

\begin{lemma}\label{lem.hyperflaechenformel}
Let $\Joker$ be one of the bundle maps defined above.
For $X\in \Gamma(T_MN)$, and $\phi\in \Gamma(\Sigma_MN)$ resp.\ $\Gamma(\pSigma_MN)$ we have
\begin{equation}\label{bernd.habil.formel}
  \nabla^N_X\Joker(\phi) = \Joker(\nabla^{M,g_s}_X\phi) + \frac12 W(X)\star \nu \star \Joker(\phi),
\end{equation}
where the Weingarten map $W$ is defined through 
$g(W(X),Y)\nu=\nabla_X^NY-\nabla_X^M Y$.
\end{lemma}

A proof will be given in Appendix~\ref{sec.hypersurfaces}.

In combination with equations \eqref{J.id.N}, \eqref{J.id.mod.N}, \eqref{J.id.even}, and \eqref{J.id.mod.even}
we will see later in this article that families of $(M,g_s)$-parallel spinors will lead to solutions of the generalized imaginary Killing spinor equation \eqref{eq.imag.killing} which is one of the constraint equations. 


\subsection{BBGM connection}\label{subseq.bbgm} 

Let $M$ be a compact spin manifold. We denote the space of all Riemannian
metrics on $M$ by $\cM$. For every metric $g\in \cM$  we define   
$$\cF_g = \Gamma(\Sigma^gM),$$
and the disjoint union 
  $$\cF:=\coprod_{g\in \cM} \cF_g.$$
One can equip $\cF$ and $\cM$ naturally with the structure of a 
Fr\'echet bundle $\pi^\cF\colon\cF\to\cM$. This bundle structure is needed to define a connection on this bundle. 
The oldest references that we used are by Bourguignon and Gauduchon \cite{bourguignon.gauduchon:92} resp.\ by Bär, Gauduchon and Moroianu \cite{baer.gauduchon.moroianu:05}, this is why we choose the abbreviation BBGM for the four names. However we were told that there was also work by Bismut. The concepts were later properly formalized under the name 'universal spinor bundle' in \cite{ammann.weiss.witt:16a} and \cite{mueller.nowaczyk:17}, where a finite-dimensional fiber bundle with a partial connection is constructed whose sections correspond to the elements of~$\cF$ such that the parallel transport corresponds to the BBGM parallel transport. The connection is given in terms of horizontal spaces $\cH_{(g,\phi)}$, i.e.\ vector spaces satisfying
$$ T_{(g,\phi)}\cF= \cH_{(g,\phi)}\oplus \Gamma(\Sigma^g M)$$
in the sense of topological vector spaces such that 
  $$d\pi^\cF|_{(g,\phi)}:T_{(g,\phi)}\cF\to T_g\cM=\Gamma(\bigodot\nolimits^2 T^*M)$$ 
restricts to an isomorphism  $I_{(g,\phi)}:\cH_{(g,\phi)}\to \Gamma(\bigodot\nolimits^2 T^*M)$. In other words we obtain an injective map of Fr\'echet spaces 
$L_{(g,\phi)}:\Gamma(\bigodot\nolimits^2 T^*M)\to T_{(g,\phi)}\cF$
by postcomposing  $(I_{(g,\phi)})^{-1}$ with the inclusion of  $\cH_{(g,\phi)}\subset
 T_{(g,\phi)}\cF$. 

The space $\cH_{(g,\phi)}$ will smoothly depend on $(g,\phi)$ and will be compatible with the vector bundle structure of $\cF$. We do not require more knowledge about Fr\'echet manifolds in our article, so we do not introduce this Fr\'echet structure in more detail.
 
To describe the horizontal space precisely we give the maps 
$L_{(g,\phi)}$:
We assume that $h\in  T_{g}\cM=\Gamma(\bigodot\nolimits^2 T^*M)$. Let $g_t$ be a smooth path of metrics such that $g_0=g$ and $\frac{d}{dt}|_{t=0} g_t=h$.
We consider the cylinder $Z=M\times [0,1]$ with the metric $g_t+dt^2$. Then, similar to \cite{baer.gauduchon.moroianu:05} we can extend the spinor $\phi_0:=\phi$, defined on $M\times\{0\}$, uniquely to a spinor $\phi_\bullet$ 
on $M\times [0,1]$ satisfying $\nabla_{\partial_t}\phi_\bullet=0$ where $\nabla$ denotes the Levi-Civita connection on the cylinder. Note that $\phi_\bullet$ may be interpreted as a family, parametrized over $[0,1]$, of spinor fields $\phi_t$ with respect to the family of metrics $g_t$.
 The restriction of $\phi_\bullet$ to $M\times\{t\}$, denoted by $\phi_t$ is a spinor for the metric $g_t$, and $t\mapsto \phi_t$ is a smooth path in $\cF$. We now define 
  $$L_{(g,\phi)}(h)=\frac{d}{dt}\Big|_{t=0}\phi_t.$$
\begin{lemma}
The map $L_{(g,\phi)}$ is well-defined, linear, and smooth in $(g,\phi)$.
\end{lemma}
The lemma follows from the construction of the universal spinor bundle given in \cite{mueller.nowaczyk:17}. The strategy in that paper is as follows: 
Let $\pi^S:SM\to M$ be the bundle of positive definite symmetric bilinear forms on $TM$. A complex vector bundle  $\pi^\univ:\Sigma M \rightarrow SM$
is constructed, called the \emph{universal spinor bundle}, which carries a scalar product, a Clifford multiplication with vectors in $TM$ and partial connection
on $\Sigma M$ with respect to $\pi^S\colon SM\to M$. The Clifford multiplication is given by a bilinear map $\cl_g:T_{\pi^S(g)}M\times \Sigma_g M\to \Sigma_g M$ for every $g\in SM$. Note that $g$ is a scalar product on one single tangent space, namely on $T_{\pi^S(g)}M$, and thus $\Sigma_g$ is thus a fiber of the vector bundle $\pi^\univ:\Sigma M\to SM$. This Clifford multiplication shall satisfy the Clifford relations and shall depend smoothly on $g$.
 By ``partial connection'' we mean that for any section 
$\phi$ of the bundle  $\pi^\univ:\Sigma M \rightarrow SM$, the covariant derivative
$\nabla_X\phi$ is defined for some $X\in T(SM)$ if and only $d\pi^S(X)=0$ (i.e., $X$ is vertical for $\pi^S$). This partial connection comes from the vertical connection defined in \cite[Definition~2.10]{mueller.nowaczyk:17}, and allows to define a map  $L_{(g,\phi)}$ as above. In particular this shows that $\frac{d}{dt}|_{t=0}\phi_t$ does not depend on how we choose~$g_t$, but only its derivative a $t=0$.


\section{Variation of the parallel spinor equation}\label{var.par.spinor.eq}

Let us first recall a result from \cite{ammann.weiss.witt:16a}:
For $(g,\phi) \in \cF$ let $\hat\kappa_{g,\phi}: \cH_{(g,\phi)} \rightarrow \Gamma( \Sigma^gM \otimes T^*M)$ be the map defined by

\begin{eqnarray}\label{kappa-formula}
  \hat\kappa_{g,\phi}(h) &:=&  \tfrac{1}{4} \sum_{i \neq j} (\nabla^g_{e_i} h)(\,\cdot\,\,,e_j) e_i \cdot e_j \cdot \phi.
\end{eqnarray}
In \cite[Sec. 4.3]{ammann.weiss.witt:16a} a similar map was defined, namely $\kappa_{g,\phi}: T_{(g,\phi)} \cF=\cH_{(g,\phi)}\oplus \Gamma(\Sigma^g M)\rightarrow \Gamma( \Sigma^gM \otimes T^*M)$ which is related to $\hat\kappa_{g,\phi}$ by the formula  $\kappa_{g,\phi}(h,\psi)=\hat\kappa_{g,\phi}(h)+\na_\cdot^g\psi$. It was shown in \cite[Lemma~4.12]{ammann.weiss.witt:16a} that  this is related to the differential of the map $K_X:\cF \to \cF$, $(g,\phi)\mapsto (g,\nabla^g_X\phi)$
as follows: 
  \begin{eqnarray}\label{kappa-formula2}  
  (d_{(g,\phi)} K_X)(h, \psi) =  (h, \kappa_{g,\phi}(h,\psi)(X)) = (h, \hat\kappa_{g,\phi}(h)(X)+ \nabla^g_X \psi).
  \end{eqnarray}
We define the Wang map $\cW_{g,\phi}$ as the composition
$$\bigodot\nolimits^2 T^*M\hookrightarrow\bigotimes\nolimits^2 T^*M\xrightarrow{\cl^g(\,\cdot\,,\phi)\otimes \id} \Sigma^gM\otimes T^*M$$
so that
\begin{align*}
\cW_{g,\phi}(h)=(X\mapsto\sum_{j=1}^n h(X,e_j)e_j\cdot \phi ).
\end{align*}
It is easy to see (c.\ f.\ \cite[Lemma 2.3, 1.]{dai.wang.wei:05}) that
  $$|\cW_{g,\phi}(h)|=|\phi|\cdot|h|$$
and
  $$\nabla_X \cW_{g,\phi}(h) = \cW_{g,\phi}(\nabla_X h) + \cW_{g,\nabla_X\phi}(h).$$

\begin{lemma}\label{lem.kappa.einstein}If $\phi$ is a parallel spinor on $(M,g)$, then 
  the diagram
  \begin{center}
   \begin{tikzpicture}[node distance=2.8cm, auto]
    \node (TT1) at (0,0) {$\bigodot\nolimits^2 T^*M$};
    \node (twD1) at (4,0) {$\Gamma(\Sigma^gM\otimes T^*M)$};
    \node (twD2) at (4,-2) {$\Gamma(\Sigma^gM\otimes T^*M)$};
    \draw[->] (TT1) to node {$\cW_{g,\phi}$} (twD1);
    \draw[->] (TT1) to node [swap] {$4\hat\kappa_{g,\phi}+\phi\otimes \divv^g$} (twD2);
    \draw[->] (twD1) to node {$D^{T^*M}$} (twD2);
   \end{tikzpicture}
   \end{center}
commutes. In other words we have for any $h\in \bigodot\nolimits^2 T^*M$
  $$D^{T^*M}\circ \cW_{g,\phi} (h)= 4\hat\kappa_{g,\phi}(h) + \phi\otimes \divv^g(h),$$
  
where, for a vector bundle $E$ over $M$ with connection we define the Dirac operator $D^E$ on $\Sigma^gM \otimes E$ by $\cl^g \circ \nabla^{\Sigma^gM \otimes E}$ where $\nabla^{\Sigma^gM \otimes E} $ is the product connection.
\end{lemma}

\begin{proof}
  Let $h \in \Gamma(\bigodot^2 T^*M)$.
  Then $\cW_{g,\phi}(h)=\sum_{j=1}^n h(\,\cdot\,\,,e_j)e_j\cdot \phi$.
  We now assume that $(e_1,\ldots,e_n)$ is a local orthonormal frame satisfying $\nabla e_j|_p=0$ at $p\in M$ for $1\leq j\leq n$, and we calculate in $p$
  \begin{align*}D^{T^*M}\left(\cW_{g,\phi}(h)\right) &= \sum_{i=1}^n e_i\cdot \nabla_{e_i}\left(\cW_{g,\phi}(h)\right)\\
    &= \sum_{i,j=1}^n(\nabla_{e_i}h)(\,\cdot\,\,,e_j) e_i\cdot e_j\cdot \phi\\
    &= 4 \hat\kappa_{g,\phi}(h) +  \sum_{i=1}^n(\nabla_{e_i}h)(\,\cdot\,\,,e_i) e_i\cdot e_i\cdot \phi\\
    &= 4 \hat\kappa_{g,\phi}(h) +  \phi\otimes \underbrace{\sum_{i=1}^n(-\nabla_{e_i}h)(\,\cdot\,\,,e_i)}_{=\divv^g h}.
  \end{align*}
\end{proof}

Very similarly we prove for arbitrary sections $\phi$ and $h$:

\begin{align}\begin{split}\label{DW.formula}
  D^{T^*M}\left(\cW_{g,\phi}(h)\right)(X) &= 4 \hat\kappa_{g,\phi}(h)(X) +  (\divv^g h)(X)\phi \\&\quad - 2 \nabla_{h^\#(X)}\phi - (h^\#(X))\cdot D^g\phi,
  \end{split}
\end{align}
where
$h^\#\in \End(TM)$ is defined as $h^\#(X):=h(\,\cdot\,,X)^\#$.

The following lemma is a straightforward generalization of a Lemma by McKenzie Wang 
\cite[Lemma~3.3 (d) for $k=1$]{wang_m:91}, see also \cite[Prop.\ 2.4]{dai.wang.wei:05}.
\begin{lemma}\label{lem.dww}
For $\phi\in \Gamma(\Sigma^gM)$ and $h\in\Gamma(\bigodot\nolimits^2 T^*M)$
we have 
\begin{align*}
  (D^{T^*M})^2\cW_{g,\phi}(h)&=  \cW_{g,\phi}(\Delta_E h) - 2 \sum_\ell \cW_{g,\nabla_{e_\ell}\phi}(\nabla_{e_\ell}h)+\cW_{g,D^2 \phi}(h)\\
  &\quad-2(X\mapsto\sum_{\ell}h^{\sharp}(e_{\ell})\cdot R_{X,e_{\ell}}\phi).
\end{align*}
In particular, if $\phi$ is a parallel spinor, then the diagram
  \begin{center}
   \begin{tikzpicture}[node distance=2.8cm, auto]
    \node (TT1) at (0,0) {$\Gamma(\bigodot\nolimits^2 T^*M)$};
    \node (TT2) at (0,-2) {$\Gamma(\bigodot\nolimits^2 T^*M)$};
    \node (twD1) at (4,0) {$\Gamma(\Sigma^gM\otimes T^*M)$};
    \node (twD2) at (4,-2) {$\Gamma(\Sigma^gM\otimes T^*M)$};
    \draw[->] (TT1) to node {$\cW_{g,\phi}$} (twD1);
    \draw[->] (TT2) to node {$\cW_{g,\phi}$} (twD2);
    \draw[->] (TT1) to node [swap] {$\Delta_E$} (TT2);
    \draw[->] (twD1) to node {$(D^{T^*M})^2$} (twD2);
   \end{tikzpicture}
   \end{center}
commutes. 
\end{lemma}

\begin{proof}
	We can locally write 
\begin{align*}
\cW_{g,\phi}(h)=\sum_{i,j}h_{ij}e_j\cdot \phi\otimes e_i^*
\end{align*}
where $\left\{e_1,\ldots e_n\right\}$ is a local orthonormal frame with respect to~$g$ with $\nabla e_i=0$ at~$p$.
We define $\na_kh_{ij}\coloneq (\na_{e_k} h)(e_i,e_j)$,  $\na^2_{k,l}h_{ij}\coloneq (\na^2_{e_k,e_l} h)(e_i,e_j)$, and let  $R_{e_k,e_l}h_{ij}=\na^2_{k,l}h_{ij}-\na^2_{l,k}h_{ij}$ be the associated curvature.
By using the Clifford relations, we get on the domain of the frame:
\begin{align*}
D^{T^*M}(\cW_{g,\phi}(h))&=\sum_{i,j,k}(\nabla_kh_{ij}e_k\cdot e_j\cdot\phi+h_{ij}e_k\cdot e_j\cdot\nabla_{e_k}\phi)\otimes e_i^*\\
&=\sum_{i,j,k}(\nabla_kh_{ij}e_k\cdot e_j\cdot\phi-h_{ij}e_j\cdot e_k\cdot\nabla_{e_k}\phi-2h_{ij}\delta_{kj}\nabla_{e_k}\phi)\otimes e_i^*\\
&=\underbrace{\sum_{i,j,k}\nabla_kh_{ij}e_k\cdot e_j\cdot\phi\otimes e_i^*}_{=:(\mathrm{A})} 
+ \underbrace{(-\cW_{g,D\phi}(h))}_{=: (\mathrm{B})}
 + \underbrace{(-2\sum_{i,j}h_{ij}\nabla_{e_j}\phi\otimes e_i^*)}_{=: (\mathrm{C})}
\end{align*}
Now we apply $D^{T^*M}$ to this equation and we use the notation $\psi=D\phi$. Then we get from the previous equation at the point~$p$
\begin{align*}
D^{T^*M}((\mathrm{B}))&=-D^{T^*M}(\cW_{g,\psi}(h))\\&=-\sum_{i,j,k}\nabla_kh_{ij}e_k\cdot e_j\cdot\psi\otimes e_i^*+\cW_{g,D\psi}(h)+2\sum_{i,j}h_{ij}\nabla_{e_j}\psi\otimes e_i^*.
\end{align*}
Moreover,
\begin{align*}
D^{T^*M}((\mathrm{C}))&=-2\sum_{i,j,k}(h_{ij}e_k\cdot \nabla^2_{e_k,e_j}\phi+\nabla_kh_{ij}e_k\cdot \nabla_{e_j}\phi)\otimes e_i^*\\
&=-2\sum_{i,j,k}(h_{ij}e_k\cdot\nabla^2_{e_j,e_k}\phi
+h_{ij}e_k\cdot R_{e_k,e_j}\phi+\nabla_kh_{ij}e_k\cdot\nabla_{e_j}\phi)  \otimes e_i^*\\
&=-2\sum_{i,j}(h_{ij}\nabla_{e_j}\psi
+\sum_kh_{ij}e_k\cdot R_{e_k,e_j}\phi+\sum_k\nabla_kh_{ij}e_k\cdot\nabla_{e_j}\phi)  \otimes e_i^*
\end{align*}
and
\begin{align*}
D^{T^*M}((\mathrm{A}))&=\underbrace{\sum_{i,j,k,l}\nabla^2_{lk}h_{ij}e_l\cdot e_k\cdot e_j\cdot \phi\otimes e_i^*}_{=: (\mathrm{E})}+\underbrace{\sum_{i,j,k,l}\nabla_kh_{ij}e_l\cdot e_k\cdot e_j\cdot\nabla_{e_l}\phi\otimes e_i^*}_{=: (\mathrm{F})}.
\end{align*}
By using the relation $e_l\cdot e_k\cdot e_j= e_k\cdot e_j\cdot e_l+2e_k\delta_{lj}-2\delta_{kl}e_j$,
\begin{align*}
(\mathrm{F})=\sum_{i,j,k}(\nabla_kh_{ij}e_k\cdot e_j\cdot\psi+2\nabla_kh_{ij}e_k\cdot\nabla_{e_j}\phi
-2\nabla_kh_{ij}e_j\cdot\nabla_{e_k}\phi)\otimes e_i^*.
\end{align*}
By adding up, a lot of terms cancel and we are left with
\begin{align*}
(D^{T^*M})^2(\cW_{g,\phi}(h))&=(\mathrm{E})+\cW_{g,D^2\phi}(h)-2\sum_{i,j,k}\nabla_kh_{ij}e_j\cdot\nabla_{e_k}\phi\otimes e_i^*\\
 &\quad+2\sum_{i,j,k}h_{ij}e_k\cdot R_{e_j,e_k}\phi\otimes e_i^*
\end{align*}	
so it remains to consider the term $(\mathrm{E})$. We have
\begin{align*}
(\mathrm{E})=&-\sum_{i,j,k}\nabla^2_{kk}h_{ij}e_j\cdot \phi\otimes e_i^*+\sum_{\substack{i,j,k,l\\k\neq l}}\nabla^2_{lk}h_{ij}e_l\cdot e_k\cdot e_j\cdot \phi\otimes e_i^*\\
=&\sum_{i,j}(\nabla^*\nabla h)_{ij}e_j\cdot \phi\otimes e_i^*+\sum_{i,j,k,l}\frac{1}{2}R_{e_k,e_l}h_{ij}e_k\cdot e_l\cdot e_j\cdot\phi\otimes e_i^*.
\end{align*}
By the formula which expresses the curvature of $T^*M\otimes T^*M$ in terms of the curvature of $TM$, we obtain $R_{e_k,e_l}h_{ij}=-\sum_{r}(R_{klir}h_{rj}+R_{kljr}h_{ir})$. 
Using the relation $e_k\cdot e_l\cdot e_j=e_j\cdot e_k\cdot e_l-2e_k \delta_{jl}+2e_l\delta_{jk}$, we get
\begin{align*}
  -\frac{1}{2}\sum_{i,j,k,l,r}&R_{klir}h_{rj}e_k\cdot e_l\cdot e_j\phi\otimes e_i^*\\
  &=
-\frac{1}{2}\sum_{i,j,k,l,r}R_{klir}h_{rj}e_j\cdot e_k\cdot e_l\cdot \phi\otimes e_i^*\\
&\quad
+\sum_{i,k,l,r}(R_{klir}h_{rl}e_k-R_{klir}h_{rk}e_l)\cdot \phi\otimes e_i^*\\
&=-\frac{1}{2}\sum_{i,j,k,l,r}R_{klir}h_{rj}e_j\cdot e_k\cdot e_l\cdot \phi\otimes e_i^*-2
\sum_{i,j}(\mathring{R}h)_{ij}e_j\cdot \phi\otimes e_i^*
\end{align*}
and
\begin{align*}
  -\frac{1}{2}\sum_{i,j,k,l,r}&R_{kljr}h_{ir}e_k\cdot e_l\cdot e_j\cdot\phi\otimes e_i^*=\\
  &=
-\frac{1}{2}\sum_{i,j,k,l,r}R_{kljr}h_{ir}e_j\cdot e_k\cdot e_l\cdot \phi\otimes e_i^*\\
&\quad
+\sum_{i,k,l,r}(R_{kllr}h_{ir}e_k-R_{klkr}h_{ir}e_l)\cdot \phi\otimes e_i^*\\
=&\quad
-\frac{1}{2}\sum_{i,j,k,l,r}R_{kljr}h_{ir}e_j\cdot e_k\cdot e_l\cdot \phi\otimes e_i^*
+2\sum_{i,j,r}\Ric_{jr}h_{ir}e_j\cdot \phi \otimes e_i^*.
\end{align*}
We get
\begin{align*}
(D^{T^*M})^2(\cW_{g,\phi}(h))&=\cW_{g,\phi}\bigl(\underbrace{\nabla^*\nabla h-2\mathring{R}h }_{=\Delta_Eh}\bigr)-\frac{1}{2}\sum_{i,j,k,l,r}R_{klir}h_{rj}e_j\cdot e_k\cdot e_l\cdot \phi\otimes e_i^*\\
&\quad-\frac{1}{2}\sum_{i,j,k,l,r}R_{kljr}h_{ir}e_j\cdot e_k\cdot e_l\cdot \phi\otimes e_i^*\\
&\quad
+2\sum_{i,j,r}\Ric_{jr}h_{ir}e_j\cdot \phi \otimes e_i^*+\cW_{g,D^2\phi}(h)\\
&\quad-2\sum_{i,j,k}\nabla_kh_{ij}e_j\cdot\nabla_{e_k}\phi\otimes e_i^*+2\sum_{i,j,k}h_{ij}e_k\cdot R_{e_j,e_k}\phi\otimes e_i^*.
\end{align*}
From the standard identities
\begin{align*}
R_{e_i,e_j}\phi=\frac{1}{4}\sum_{k,l}R_{ijkl}e_k\cdot e_l\cdot\phi,\qquad
\sum_ie_i\cdot R_{e_i,e_j}\phi=\frac{1}{2}\Ric(e_j)\cdot \phi,
\end{align*}
we deduce
\begin{align*}
-\frac{1}{2}\sum_{i,j,k,l,r}R_{klir}h_{rj}e_j\cdot e_k\cdot e_l\cdot \phi\otimes e_i^*
&=-2\sum_{i,j,r}h_{rj}e_j\cdot R_{e_i,e_r}\varphi\otimes e_i^*\\
-\frac{1}{2}\sum_{i,j,k,l,r}R_{kljr}h_{ir}e_j\cdot e_k\cdot e_l\cdot \phi\otimes e_i^*&
=-2\sum_{i,j,r}h_{ri}e_j\cdot R_{e_j,e_r}\varphi\otimes e_i^*\\&
=-\sum_{i,j,r}h_{ri}\Ric_{jr}e_j\cdot \varphi\otimes e_i^*\\
2\sum_{i,j,k}h_{ij}e_k\cdot R_{e_j,e_k}\phi\otimes e_i^*&=-\sum_{i,j,k}h_{ij}\Ric_{jk}e_k\cdot\varphi\otimes e_i^*
\end{align*}

which yields the final result.
%
\end{proof}

\section{The BBGM parallel transport preserves parallel spinors}\label{sec.par.preserved}

\begin{proposition}\label{prop.bg.par.par}
  Let $I\subset\R$ be an interval, $(g_s)_{s\in I}$
  be a path of Ricci-flat metrics with a divergence free derivative, $0\in I$,
  and let $\phi_0$ be a parallel spinor on $(M,g_0)$.
  Let, for all $s \in I$, $\phi_s$ be a spinor on $(M,g_s)$ such that $\frac{d}{ds}\phi_s=0$.
  Then  $\phi_s$ is a parallel spinor on  $(M,g_s)$ for all $s\in I$.
\end{proposition}

Here $\frac{d}{ds}\phi_s=0$ should be understood as a derivation with respect to the BBGM connection.

In the proof of Proposition \ref{prop.bg.par.par}, we need a statement about continuous dependence of eigenvalues of the Dirac operator on the metric.
\begin{theorem}\label{Nicolais_thm} There exists a family of functions $\lambda_j:\mathcal{M}\to\R$, $j\in \Z$ such that 
\begin{itemize}
\item $\left\{\lambda_j(g)\right\}_{j\in \Z}=\mathrm{spec}(D_g)$ for all $g\in \mathcal{M}$
\item The family is nondecreasing, i.e.\ $\lambda_i\leq \lambda_j$ whenever $i\leq j$.
\item $\sup_{j\in \Z}|\mathrm{arsinh}(\lambda_j(g))-\mathrm{arsinh}(\lambda_j(\tilde{g})|\leq
 C\left\|g-\tilde{g}\right\|_{C^1}$
\end{itemize}
\end{theorem}
The theorem in the above version is proven in detail in \cite[Main Theo\-rem~2]{Nowaczyk2013}, but the statement we need was well-known long before.
  
To prove the proposition, we will at first prove the same statement under more restrictive assumptions.

\begin{lemma}\label{lem.bg.par.par}
  Let $(g_s)_{s\in I}$
  be a path of Ricci-flat metrics with a divergence free derivative, 
  $0\in I$,
  and let $\phi_0$ be a parallel spinor on $(M,g_0)$. 
  Assume that $\dim \ker D^{T^*M}$ is constant along $I$.
  Let $\phi_s$ be a spinor on $(M,g_s)$ such that $\frac{d}{ds}\phi_s=0$.
  Then  $\phi_s$ is a parallel spinor on  $(M,g_s)$ for all $s\in I$.
\end{lemma}

\begin{proof}[Proof of Lemma~\ref{lem.bg.par.par}]
Without loss of
generality we can assume that $I$ is compact. Using Rayleigh-Ritz type arguments, explained e.g.\ in detail in \cite{Nowaczyk2013}, one can show the existence of continuous functions $I\ni s\mapsto \lambda_i(s)\in \mR$, $\lambda_i(s)\leq \lambda_{i+1}(s)$ such that $(\lambda_i(s))_{i \in \mZ}$ are the eigenvalues (including the correct multiplicity) of $D^{T^*M}$.
If the kernel of  $D^{T^*M}$ has constant dimension $k$, then we can assume $\lambda_0(s)< \lambda_1(s)=\cdots=\lambda_k(s)=0<\lambda_{k+1}(s)$. Then for
  $$\gamma:=\min_{s\in I}\min\{-\lambda_0(s),\lambda_{k+1}(s)\}>0$$
the operator $D^{T^*M}_{g_s}$ has no eigenvalue in 
$(-\gamma,0)\cup (0,\gamma)$. Because of this, there are bounded operators 
\begin{equation}\label{D.is.an.iso}
\bigl(D^{T^*M}_{g_s}\bigr)^{-1}:L^2_{g_s}\cap \ker(D^{T^*M}_{g_s})^{\perp}\to  L^2_{g_s}\cap \ker(D^{T^*M}_{g_s})^{\perp}
\end{equation}
that invert $D^{T^*M}_{g_s}$ on the orthogonal complement of the kernel and that are uniformly bounded by $\gamma^{-1}$.
By elliptic theory, see e.g.\ \cite[Chapter III \textsection 5]{lawson.michelsohn:89} we have isomorphisms
\begin{equation}
D^{T^*M}_{g_s}:H^1_{g_s}\cap \ker(D^{T^*M}_{g_s})^{\perp}\to  L^2_{g_s}\cap \ker(D^{T^*M}_{g_s})^{\perp}.
\end{equation}
Because of $\mathrm{im}(D^{T^*M}_{g_s})\subset \ker(D^{T^*M}_{g_s})^{\perp}$, there is a constant $C_0>0$ 
such that for any $s\in I$ and for any $\sigma\in \Gamma(\Sigma^{g_s}M\otimes T^*M)$ 
we have
\begin{equation}\label{von.Dzwei.zu.D}
\|D^{T^*M}_{g_s}\sigma\|_{H^1_{g_s}}\leq C_0\|\left(D^{T^*M}_{g_s}\right)^2\sigma\|_{L^2_{g_s}}.
\end{equation}
Moreover, $g'_s\in \ker(\Delta_{E,g_s})$ by the facts collected in Subsection \ref{conventions}.
We calculate using Lemma~\ref{lem.dww}
\begin{align*}
  (D^{T^*M}_{g_s})^2\cW_{g_s,\phi_s}(g'_s)= & - 2 \sum_\ell \cW_{g_s,\nabla_{e_\ell}\phi_s}(\nabla_{e_\ell} g'_s)+\cW_{g_s, D^2 \phi_s}(g'_s)\\
  &-2(X\mapsto\sum_{\ell}({g}'_s)^{\sharp}(e_{\ell})\cdot R_{X,e_{\ell}}\phi_s),
\end{align*}
thus pointwise  $(D^{T^*M}_{g_s})^2\cW_{g_s,\phi_s}( g'_s)$ is a linear expression
in $\nabla \phi_s$ and its first derivative. Thus
  $$ \left\|\big(D^{T^*M}_{g_s}\bigr)^2 \cW_{(g_s,\phi_s)}(g'_s) \right\|_{L^2} \leq C \|\nabla\phi_s\|_{H^1},$$
  where $C$ depends on $\sup_{s\in I}\left\|{g}'_s\right\|_{C^1(g_s)}$.
 This implies, using Eq. \eqref{von.Dzwei.zu.D}:

\begin{equation}\label{Absch.D.nabla.phi}
\left\|D^{T^*M}_{g_s} \cW_{(g_s,\phi_s)}(g'_s) \right\|_{H^1}  \leq C C_\gamma\|\nabla\phi_s\|_{H^1}. 
\end{equation}

  We now differentiate $\nabla^{g_s}\phi_s$ using formulae \eqref{kappa-formula2} and \eqref{DW.formula} and Lemma~\ref{lem.kappa.einstein}.
  \begin{align*}
    \frac{d}{ds}&  \nabla^{g_s}\phi_s = \hat\kappa_{g_s,\phi_s}( g_s') +\nabla^{g_s}  \underbrace{\phi'_s}_{=0}\\
    & = \frac14\Biggl(D^{T^*M}_{g_s}\Bigl(\cW_{(g_s,\phi_s)}(g'_s)\Bigr)   - \phi_s\otimes \underbrace{\divv^{g_s}  g_s'}_{=0}  + 2 \nabla_{(g_s')^\#(\,\cdot\,)}\phi_s +  (g_s')^\#(\,\cdot\,)\cdot D^{g_s}\phi_s\Biggr)\\
   & =  \frac14 D^{T^*M}_{g_s} \cW_{(g_s,\phi_s)}( g_s')+\frac12 \nabla_{( g_s')^\#(\,\cdot\,)}\phi_s + \frac14  (g_s')^\#(\,\cdot\,)\cdot D^{g_s}\phi_s.
  \end{align*}
This implies, with Eq. \ref{Absch.D.nabla.phi}:
\begin{align*}\left\|\frac{d}{ds}  \nabla^{g_s}\phi_s\right\|_{H^1(g_s)} & \leq \tilde{C}\Biggl( \|\nabla^{g_s}\phi_s\|_{H^1(g_s)} + \underbrace{\|D^{g_s}\phi_s\|_{H^1(g_s)}}_{\leq \sqrt{n}\|\nabla^{g_s}\phi_s\|_{H^1(g_s)}}\Biggr)
\leq C_1 \|\nabla^{g_s}\phi_s\|_{H^1(g_s)}
\end{align*}
where $\tilde{C}$ and $C_1$ depend on $\sup_{s\in I}\left\|{g}'_s\right\|_{C^1(g_s)}$.

We set $\ell(s):=\left\| \nabla^{g_s}\phi_s\right\|_{H^1(g_s)}^2$. In order 
to derive $\ell(s)$ one has to be aware that also the evaluating metric depends
on $s$. The dependence on the metrics effects $\ell(s)$ in the metric contractions, the volume element and in the covariant derivatives used to define by $\|\sigma\|_{H^1(g)}^2=
\|\sigma\|_{L^2(g)}^2+ \|\nabla^g\sigma\|_{L^2(g)}^2$. Due to compactness of $I$, these effects lead to a number $C_2>0$, constant in $s$, such that 
$$|\ell'(s)- 2 \langle \frac{d}{ds}  \nabla^{g_s}\phi_s, \nabla^{g_s}\phi_s\rangle_{H^1(g_s)}|\leq C_2 \ell(s)$$
 and we get
  $$\ell'(s)\leq  2 |\langle \frac{d}{ds}  \nabla^{g_s}\phi_s, \nabla^{g_s}\phi_s\rangle_{H^1(g_s)} |+C_2\ell(s) \leq  (2C_1 +C_2)\,\ell(s).$$
As $\ell(0)=0$ this implies with Gr\"onwall's inequality that $\ell(s)$ 
vanishes for all~$s$.
\end{proof}

\begin{rem}
In the proof above the derivative $\frac{d}{ds}\nabla^{g_s}\phi_s$ should be taken with some care. Here we derive an $s$-dependent family $\nabla^{g_s}\phi_s\in \Gamma(T^*M\otimes \Sigma^{g_s}M)$ with respect to the metric variation given by $s$. This is the BBGM-derivative in the spinorial part. On the cotangential part, one could also use a BBGM-kind of derivative, but this is not what was used. On the cotangential part, we simply used the derivative in the usual sense, i.e. the derivative of a curve in the vector space $\Gamma(T^*M)$.
\end{rem}

\begin{proof}[Proof of Proposition~\ref{prop.bg.par.par}]
In a first step we prove the Proposition for analytic families 
$g_s$, $s\in I$ of Ricci-flat metrics, in other words we assume that the map
$I\to \bigodot^2T^*M$ defined by  $g_s$ is analytic. 

Then $D^{T^*M}_{g_s}$ is an analytic family of operators. This implies that the eigenvalues $\lambda_i(s)$ of $D^{T^*M}_{g_s}$ can be numbered such that $s\mapsto\lambda_i(s)$ is analytic in $s$ \cite[Appendix A]{hermann.diss}.
Due to Theorem \ref{Nicolais_thm}, only finitely many $\lambda_i(s)$ have zero sets on a compact interval. Let $\mu:= \min \dim \ker D^{T^*M}_{g_s}$, and
  $J:=\{s\in I\mid \dim \ker D^{T^*M}_{g_s}>\mu\}$ which is a closed discrete subset due to analyticity. (We conjecture that $J=\emptyset$, but we were unable to prove it.)
%
Let $I_0$ be a connected component of $I\setminus J$. The previous lemma states
that the Bourguignon-Gauduchon parallel transport along $s\mapsto g_s$ 
maps parallel spinor $\phi\in\Gamma(\Sigma^{g_s}M)$ to a parallel spinor $P_{s,r}(\phi)\in\Gamma(\Sigma^{g_r}M)$ for any $s,r\in I_0$.
Let $\bar{r}\in\partial I_0$. Then by continuity,
  $$P_{s,\bar r}(\phi)= \lim_{r\to\bar r}P_{s, r}(\phi)$$
is a parallel section of $\Gamma(\Sigma^{g_{\bar r}}M)$. We use the fact that 
the dimension of the space of parallel spinors is locally constant (see \cite{ammann.kroencke.weiss.witt:18}), thus for any $s\in I_0$ and and $r\in \bar I_0$ (including $r\in \partial I_0$) the monomorphism $P_{s,r}$ from parallel spinors in $\Gamma(\Sigma^{g_s}M)$ to parallel spinors  in $\Gamma(\Sigma^{g_r}M)$ is an isomorphism.
Thus BBGM parallel transport
perserves parallel spinors along~$\bar I_0$, and by an induction argument (using that $J$ is finite in compact intervals) this
also holds along $I$. The proposition is thus proven for analytic families $s\mapsto g_s$ and thus also for piecewise analytic families $s\mapsto g_s$.

Now, as the Riemannian metrics form an open cone in the vector space of bilinear forms, an arbitrary smooth family $s\mapsto g_s$ can be approximated by piecewise analytic paths in the $C^1$-norm. 
We claim that the BBGM parallel transport is 
continuous in this limit. This can be seen most easily in the universal spinor bundle formulation:
There, for $S M$ being the bundle of symmetric positive definite bilinear forms, a Clifford bundle $\pi^\univ: \Sigma M \rightarrow SM$ is constructed such that, for each $g \in \Gamma(SM) $, $ \Sigma^gM  $ is isomorphic as a Clifford bundle  to $g^* \Sigma M $. Furthermore, the vector bundle $\pi^\univ\colon\Sigma M \rightarrow SM$ carries a vertical (w.r.t. $S M \rightarrow M$) covariant derivative   $\nabla$ whose parallel transport $P$ is linked to the BBGM parallel transport $\hat{P}$ as follows: Let $g_\bullet:[0,\ell]\to \Gamma(S M)$, $s\mapsto g_s$ be a $C^1$-curve of Riemannian metrics on $M$.
  Let $p \in M$ and let $g|_p: s \mapsto g_s (p)$ and $\sigma \in (\pi^\univ)^{-1 } (g_0(p))=\Sigma^{g_0}_p M $, then for any $\tilde{\sigma}  \in \Gamma (\Sigma^{g_0}M)$ with $\tilde{\sigma}|_p = \sigma $ we have 
  $$\hat{P}_{g_\bullet} (\tilde{\sigma})\big|_p = P_{g|_p} (\sigma).$$
  The parallel transport along a curve with respect to a connection is given by a first order ordinary differential equation (ODE), satisfying the conditions of the theorem of Picard-Lindelöf. Using the universal spinor bundle formalism we argued that the BBGM-parallel transport is given by such an ODE and that its coefficient functions converge uniformly (i.e. in the $C^0$-norm) when a path of metrics converges in the $C^1$-norm to a limit path of metrics.\footnote{To be precise: we only need control of the norm
    $$\|g_\bullet\|:=\max_{s\in [0,\ell]} \left(\|g_s\|_{C^0(M)}+ \|\frac{d}{ds}g_s\|_{C^0(M)}\right).$$}
Thus the theorem of Picard-Lindelöf, taking into account that both $M$ and $[0,\ell]$ are compact,  implies that the BBGM parallel transport converges uniformly when we approximate a smooth path of metrics $s\mapsto g_s$ by piecewise analytic paths of metrics in the $C^1$-norm.
Thus the BBGM parallel transport also preserves parallel spinors along
the smooth family $s\mapsto g_s$.
\end{proof}

\section{Construction of solutions to the constraint equations}\label{sec.constr.sol}

In this section we want to use the BBGM connection to construct solutions of the constraint equation on suitable manifolds of the form $(M\times I, g_s+ds^2 )$ where~$I$ is an interval and where~$M$ is an $m=(n-1)$-dimensional manifold.

Assume that $g_s$ is a family of Ricci-flat metrics with divergence-free derivative and that for some $s_0$ there is a non-trivial parallel spinor $\phi_{s_0}$ on $(M,g_{s_0})$. By rescaling we can achieve that its norm is~$1$ in every point. We 
shift it in the $s$-direction parallely with the
BBGM parallel transport. By Prop. \ref{prop.bg.par.par} we obtain a family $\phi_s$ 
of $g_s$-parallel spinors of constant norm~$1$ on $M$, for every $s\in I$. This yields a fiberwise parallel section of $\Sigma_M^{(\#)}N$.  
Recall $\nu=\frac{\partial}{\partial s}$. 


In the following we assume that $f:I\to \mR$ is a given smooth function and choose $s_0\in I$. We define $F(s):=\exp \left(-\frac12\int_{s_0}^s f(\sigma)\,d\sigma\right)$, i.e.\ $-\frac12 f(s)=F'(s)/F(s)$.

\textbf{Case $n$ is odd, i.e.\ $m=\dim M$ even.} By possibly changing the orientation and using Proposition~\ref{prop.change.or} in Appendix~\ref{sec.change.of.orientation} we can assume that all $\phi_s\in \Gamma(\Sigma^{g_s} M)$
have the positive parity for the splitting given by the volume element of $M$,
i.e. we can assume
$\omega_\C^M\cdot\phi_s=\phi_s$.
For every $(x,s)\in N=M\times I$ we use the map~$J$ from Section~\ref{subsec.hyp.surf} to get identifications $\Sigma_xM\cong {\Sigma_M}_{(x,s)} N\cong \Sigma_{(x,s)} N$. Again the Clifford multiplication for $(M,g_s)$ will be denoted by $\cdot$, and for the one for $(N,h)$, $h=g_s+ds^2$ we will use $\star$.
%
%
%
From Equation~\eqref{J.id.N} we get  $\nu\star \phi=i\phi$. Using Lemma~\ref{lem.hyperflaechenformel} we obtain
\begin{eqnarray*}
  \nabla^N_{\nu}\phi&=&0,\\
  \nabla^N_X\phi &=& \frac12 W(X)\star \nu\star \phi = \frac{i}2 W(X)\star \phi.
\end{eqnarray*}
We set 
\begin{equation}\label{def.constr.spinor.odd}
  \Psi(x,s):=
   F(s) \phi_s(x).
\end{equation}
Assuming that $X$ is tangent to $M$, we obtain
\begin{eqnarray*}
  \nu\star \Psi&=&i\Psi,\\
  \nabla^N_{\nu}\Psi&=&\frac{F'(s)}{F(s)}\Psi=-\frac12 f(s)\Psi= \frac{i}2 f(s) \nu\star\Psi,\\
  \nabla^N_X\Psi &=& \frac12 W(X)\star \nu\star \Psi = \frac{i}2 W(X)\star \Psi.
\end{eqnarray*}
Thus $\Psi$ is an imaginary $\ol W$-Killing spinor with $\ol W= W+f\nu^b\otimes \nu$, i.e.\ it satisfies~\eqref{eq.imag.killing}.

The first relation and the defining equation \eqref{def.constr.spinor.odd} also imply $-i\<\nu\star \Psi, \Psi\>= F^2$. For $X$ tangent to $M$, the real part of 
$\<X\star \Psi, \Psi\>$ vanishes and
$$\Im \<X\star \Psi, \Psi\>= - \Re \<X\star i\Psi, \Psi\>=- \Re \<X\star\nu\star  \Psi, \Psi\>\stackrel{(*)}{=}\Re \< \Psi, X\star\nu\star \Psi\>,$$
where in $(*)$ we used the skew-symmetry of Clifford multiplication with vectors twice and $X\star \nu=-\nu\star X$.
On the other hand one  $\Re \<X\star\nu\star  \Psi, \Phi\>=\Re \< \Phi, X\star\nu\star \Psi\>$, which in particular holds for $\Phi:=\Psi$. We thus get $\Im \<X\star \Psi, \Psi\>=0$.

Thus $- i \langle X \star \Psi, \Psi \rangle = F^2 h(\nu,X)$ for every vector $X$ and 
consequently, the Dirac current $U_\Psi$ of $\Psi$, defined by 
$h(U_\Psi, X) = - i \langle X \star \Psi, \Psi \rangle$ for every vector $X$ (see \eqref{def.dirac.current})
is $U_\Psi=F^2\nu$.
We obtain
$$U_\Psi\star \Psi = i F^2 \Psi,$$
which is the constraint equation \eqref{eq.Uphi}
for $u_\Psi:=\sqrt{h(U_\Psi,U_\Psi)} \equiv F^2$. Note that here we used \eqref{def.small.uphi} as definition for $u_\psi$.
We thus have obtained solutions of the constraint equations.

\textbf{Case $n$ even, i.e.\ $m=\dim M$ odd.} 
We use the map $J^{(\#),+}$ defined in Section~\ref{subsec.hyp.surf} to view $\Sigma_x M\cong \Sigma_{M (x,s)}N$ as a subbundle
of $\Sigma_{(x,s)} N$. In particular, $(x,s)\mapsto \phi_s(x)$ then yields a section  of constant length~$1$ of the bundle $\Sigma N\to N$.

Because of equations \eqref{J.id.even} and \eqref{J.id.mod.even} we have $\nu\star\phi= i\phi$.
Using Lemma~\ref{lem.hyperflaechenformel} we obtain 
\begin{eqnarray*}
  \nabla^N_{\nu}\phi&=&0\\
  \nabla^N_X\phi &=& \frac12 W(X)\star \nu\star \phi = \frac{i}2 W(X)\star \phi
\end{eqnarray*}
We set 
\begin{equation}\label{def.constr.spinor.even}
\Psi(x,s):=F(s) \phi_s(x).
\end{equation}
Then for $X$ tangent to $M$,
 \begin{eqnarray*}
  \nu\star \Psi&=&i\Psi,\\
  \nabla^N_{\nu}\Psi&=&\frac{F'(s)}{F(s)}\Psi=-\frac12 f(s)\Psi= \frac{i}2 f(s) \nu\star\Psi,\\
  \nabla^N_X\Psi &=& \frac12 W(X)\star \nu\star \Psi = \frac{i}2 W(X)\star \Psi.
\end{eqnarray*}
Thus, $\Psi$ is an imaginary $\ol W$-Killing spinor with $\ol W= W+f\nu^b\otimes \nu$, i.e.\ it satisfies \eqref{eq.imag.killing}. With the same arguments as in the other case, we can prove that \eqref{def.dirac.current}, \eqref{char.quad.cond}, 
\eqref{eq.Uphi}, and \eqref{def.small.uphi} are satisfied, i.e. we have found a solution to the constraint equations.
 
\begin{example}
	Let $M=T^{n-1}$, $g$ be a flat metric on $M$ and $\varphi$ a parallel spinor on it.
	\begin{itemize}
		\item Let $g_s\equiv g$ with $s\in I$. 
The BBGM parallel transport leaves $\varphi$ invariant and we obtain initial data to the Cauchy problem on the metric $g_s+ds^2$ on either $T^{n-1}\times \R$ or $T^{n-1}\times S^1 =T^{n}$. The Minkowski metric together with a parallel spinor (or in the $S^1$ case a $\mZ$-quotient of it) is then a solution of the associated Cauchy problem.
		\item Let $g_s=e^{2s}g$ with $s\in \R$. For the submanifolds $M\times \{s\}\subset N$ 
we have $W=-\id|_{TM}$ for alle $s\in \mR$. 
We take the function $f(s)=-1$, i.e.\ $F(s):=\exp(s/2)$ and $\ol W=-\id|_{TN}$.
The metric we obtain is now the hyperbolic metric $h=e^{2s}g+ds^2$ on $T^{n-1}\times \R$ together with an imaginary Killing spinor with Killing constant $-i/2$. The Lorentzian cone  $(T^{n-1}\times \R\times \mR_{>0},r^2 h-dr^2)$, where $r\in \mR_{>0}$, together with a parallel spinor solves solves the associated Cauchy problem.
Note that this cone is the quotient by a $\mZ^{n-1}$-action of $I_+(0)\subset \mR^{n,1}$, defined as the set of all future-oriented time-like vectors in the $(n+1)$-dimensional Minkowski space $\mR^{n,1}$.
\end{itemize}
\end{example}
\begin{rem}
In this example we have seen two different ways of reconstructing Lorentzian Ricci-flat metrics on quotients of subsets of Minkowski space together with a parallel spinor.  However, our construction allows many more interesting examples.
\end{rem}

\begin{example}\ 
\begin{itemize} 
  \item Any imaginary Killing spinor on a complete, connected Riemannian spin manifold arises this way. 
This was proven by Baum, Friedrich, Grunewald and Kath in \cite[Chap.~7]{baum.friedrich.grunewald.kath:91}, more precisely in Theorem~1 on page 160  and Cor. 1 on page 167 \cite[Chap.~7]{baum.friedrich.grunewald.kath:91}. 
  \item This was generalized by Rademacher \cite{rademacher:91}, see also \cite[Theorem A.4.5]{ginoux:book}. Rademacher proved that any generalized imaginary Killing spinor with $W=\alpha \id$, $\alpha\in C^\infty(N,\mR)$ arises by our construction.
\item Our construction generalizes previous constructions of imaginary Killing spinors, as e.g. \cite[Prop.~4.6 and Cor.~4.7]{ginoux.habib.raulot:15}.
\end{itemize}
\end{example}

\section{From curves in the moduli space to initial data sets}\label{sec.curves.moduli}

Let $\Diff_0(M)$ be the identity component of the diffeomorphism group of $M$, acting on the space $\mathcal{M}_\parallel(M)$ of structured Ricci-flat metrics by pullback (as usual, a semi-Riemannian manifold $(M,g)$ is called structured iff its semi-Riemannian universal covering $(\tilde{M}, \tilde{g}) $ admits a parallel spinor). Furthermore, let $\Mod_\parallel(M):=\mathcal{M}_\parallel(M)/\Diff_0(M)$ be the associated 
premoduli space. 
It was shown in \cite{ammann.kroencke.weiss.witt:18} (using previous work about the special case of simply connected manifolds with irreducible holonomy)
that $\Mod_\parallel(M)$ ``naturally'' carries the structure of a 
finite-dimensional smooth manifold. The smooth structure on $\Mod_\parallel(M)$ 
can be described by the following properties
\begin{itemize}
\item $\Mod_\parallel(M)$ carries the quotient topology, if we equip $\mathcal{M}_\parallel(M)$ and $\Diff_0(M)$ with the standard Fr\'echet topology, 
\item any smooth family $N\to \mathcal{M}_\parallel(M)$, $y\mapsto g_y$ yields a smooth map   $N\to\Mod_\parallel(M)$, 
$y\mapsto [g_y]$,
\item in the case $N=(a,b)$ one has $\frac{d}{ds}|_{s=s_0}[g_s]=0$ if and only if $\frac{d}{ds}|_{s=s_0}g_s$ is tangent to the $\Diff_0(M)$-orbit 
through $g_{s_0}$.
\end{itemize}

If $I$ is an interval, then any smooth curve $I\to \Mod_\parallel(M)$ can be written as $s\mapsto [g_s]$ with a family of metrics $\divv^{g_{s_0}}\left(\frac{d}{ds}|_{s=s_0}g_s\right)=0$, see Appendix~\ref{app.div.free}.

The main results of this article is a procedure to construct initial data for the constraint equations for Lorentzian metrics with parallel spinors. 

\begin{mainconstruction}[Initial data on an open manifold]\label{mainconstrone}
Let $I$ be an interval. For any smooth curve $I\to \Mod_\parallel(M)$ $s\mapsto [g_s]$ and every positive smooth function $F:I\to \mR_{>0}$ we obtain a solution of
the initial data equations \eqref{eq.imag.killing} and \eqref{eq.Uphi} on $M\times I$ with norm $F(t)$ at any $(x,t)\in M\times I$. 
\end{mainconstruction}
Here the metric on  $M\times I$ is given by $g_s+ds^2$ provided that  $g_s$ is chosen such that $\divv^{g_{s_0}}\frac{d}{ds}|_{s=s_0}g_s=0$ and the spinor 
is given by \eqref{def.constr.spinor.odd} resp.\ \eqref{def.constr.spinor.even}.
In particular the spinor can be normalized such that it has norm $F(t)$ at any $(x,t)\in M\times I$. As derived in the preceding section, \eqref{eq.imag.killing} are then satisfied, as well as  \eqref{def.dirac.current}, \eqref{char.quad.cond}, 
\eqref{eq.Uphi}, and \eqref{def.small.uphi}.


The situation is slightly more complicated if we want to obtain solutions of the constraint equations on a closed manifold. We start with a closed curve $S^1\cong\mR/L\mZ\to \Mod_\parallel(M)$. 
If we identify $\mR/L\mZ$ with $[0,L]/0\sim L$, then every curve $\mR/L\mZ\to \Mod_\parallel(M)$ can be written as $[0,L]\to \Mod_\parallel(M)$, $s\mapsto [g_s]$ with $\divv\left(\frac{d}{ds}g_s\right)=0$, but in general we will have $g_0\neq g_L$ although $(M,g_0)$ and $(M,g_L)$ are isometric with respect to an isometry $\zeta\in\Diff_0(M)$. 
We then glue $(M\times\{0\},g_0)$ with  $(M\times \{L\},g_L)$ isometrically using the diffeomorphism $\zeta\in\Diff_0(M)$. This yields a closed Riemannian manifold $(N,h)$ diffeomorphic to $M\times S^1$. 
In order to equip it with a spin structure we have to lift 
$d\zeta^{\otimes n}:P_{\SO}(M,g_0)\to P_{\SO}(M,g_L)$ to a map between the corresponding spin structures $\zeta_\#:P_{\Spin}(M,g_0)\to P_{\Spin}(M,g_L)$. 
This yields a spin structure and a spinor bundle on $N$. 
Let $F:S^1\to \mR_{>0}$ be given. For any parallel spinor $\phi_0$ on $(M,g_0)$ 
equation \eqref{def.constr.spinor.odd} resp.\ equation \eqref{def.constr.spinor.even} yields a generalized imaginary Killing spinor $\Psi$ on $M\times [0,L]$ as in Main Construction~\ref{mainconstrone}. The gluing described above allows to view $\phi_L:=\Psi_{M\times \{L\}}$ as a parallel spinor on $(M,g_0)$. 
However, in general $\phi_L$ will differ from $\phi_0$. Let $\Gamma_\parallel(\Sigma_{g_0}M)$ denote the space of parallel spinors on $(M,g_0)$. Then 
$\phi_0\mapsto \phi_L$ yields a unitary map  $P:\Gamma_\parallel(\Sigma_{g_0}M)\to \Gamma_\parallel(\Sigma_{g_0}M)$. The map $P$ does neither depend on $F$ nor on the parametrization of the curve $s\mapsto [g_s]$. We say that $(M,g_s,\phi_0)$ satisfies the \emph{fitting condition} if $P(\phi_0)=\phi_0$ for a suitable choice of spin structure
on $N$.
The fitting condition is always satisfied in the following cases:
\begin{enumerate}[(1)]
\item $(M,g_0)$ is a $7$-dimensional manifold with holonomy $G_2$
\item $(M,g_0)$ is an $8$-dimensional manifold with holonomy $\Spin(7)$
\item $(M,g_0)$ is a Riemannian product of manifolds of that kind and of at most one factor diffeomorphic to $S^1$.
\item finite quotients of such manifolds
\end{enumerate}

As this statement is not within the core of this article, we only sketch the proof. In the first two cases the spinor bundle $\Sigma_g M$ is the complexification of the real spinor bundle $\Sigma_g^\mR M$, and thus $\Gamma_\parallel(\Sigma_g M)=\Gamma_\parallel(\Sigma_g^\mR M)\otimes \mC$. The real spinor representations of $G_2$ and $\Spin(7)$ on $\Sigma_n^\mR$ have a 1-dimensional invariant subrepresentation, thus $\dim_\mR \Gamma_\parallel(\Sigma_g^\mR M)=1$. Thus $P$ is either $+\id$ or $-\id$, and the $+$-sign can be achieved by a suitable choice of the lift $\zeta_\#$.  On $S^1$ it follows from a direct calculation. 

The map $P$ behaves ``well'' under taking products and finite quotients, thus the other two statements follow as well. 

For manifolds with a least one factor of holonomy $\SU(k)$ or $\Sp(k)$, or also for tori of dimension $>1$, however, we expect that generically the fitting condition does not hold. In this case, we expect that the space of closed paths $s\to [g_s]$ for wich $P$ has finite order (in the sense $\exists \ell\in \mN:\, P^\ell=\id$) is dense in the space of all closed paths $s\to [g_s]$ with respect to the $C^\infty$-topology. This is in fact a consequence of work in progress by Bernd Ammann, Klaus Kröncke and Hartmut Wei\ss{}.

Then passing to an $\ell$-fold cover of~$N$ obtained from running along the path $s\to [g_s]$ not just once, but $\ell$ times, we obtain a solution of the constraint equation on $M\times S^1$.

\begin{mainconstruction}[Initial data on a closed manifold]\label{mainconsttwo}
Let $L>0$. Let $\mR/L\mZ\to \Mod_\parallel(M)$, $[s]\mapsto [g_s]$  be a smooth path and let $\phi_0\in \Gamma_\parallel(\Sigma^{g_0} M)$ be given such  that $(M,g_s,\phi_0)$
satisfies the fitting condition. 
Then for any function $F:S^1\to \mR_{>0}$ we obtain a solution of the initial data equations \eqref{eq.imag.killing} and \eqref{eq.Uphi} on $M\times S^1$. 
\end{mainconstruction}

It also seems interesting to us to allow a slight generalization of our initial
problem, by considering spin$^c$ spinors with a flat associated line bundle instead of spinors in the usual sense. Assume that $\theta\in \mC$ has norm $1$.
Identifying $(v,t)\in \mC\times \mR$ with $(\theta v,t+L)$ yields a complex line bundle $L_\theta$ over $S^1 =\mR/L\mZ$. On $L_\theta$ we choose the connection such that local sections with constant $v$ are parallel. By pull back we obtain complex line bundles with flat, metric connections on $N=M\times S^1$ and on $\ol N=N\times (-\ep,\ep)$. These line bundles will also be denoted by $L_\theta$. The bundle $\Sigma_h N\otimes L_\theta$ resp.\ $\Sigma_{\bar h} {\ol N}\otimes L_\theta$ is then a spin$^c$-spinor bundle with flat associated bundle $L_\theta$. The objects tensored by $L_\theta$ will be called $L_\theta$-twisted. All the results of this article immediately generalize to $L_\theta$-twisted spinors.
We ask for $L_\theta$-twisted parallel spinors, i.e.\ parallel sections of  $\Sigma_{\bar h} {\ol N}\otimes L_\theta$ instead of parallel spinors in the usual sense. This leads to $L_\theta$-twisted constraint equations, and the $L_\theta$-twisted Cauchy problem can be solved the same way as the untwisted. 

Let $P\in \U\bigl(\Gamma_\parallel(\Sigma_{g_0}M)\bigr)$ be as above. As $P$ is unitary, there is a basis of $\Gamma_\parallel(\Sigma_{g_0}M)$ consisting of eigenvectors of $P$ for complex eigenvalues of norm $1$.

\begin{mainconstruction}[Spin$^c$-version]\label{mainconstrthree}
Let $L>0$. Let $\mR/L\mZ\to \Mod_\parallel(M)$, $[s]\mapsto [g_s]$  be a smooth path. Let $\phi_0\in \Gamma_\parallel(\Sigma_{g_0}M)$ be an eigenvalue of $P$ to the eigenvalue $\theta$. Then for any function $F:S^1\to \mR_{>0}$ we obtain
 a solution of the $L_\theta$-twisted version of the constraint equations \eqref{eq.imag.killing} and \eqref{eq.Uphi} on $M\times S^1$. 
\end{mainconstruction}

\appendix

\section{Independence of the constraint equations}\label{appendix.indep}

In this appendix we want to show that the constraint equations \eqref{def.dirac.current}--\eqref{def.small.uphi} presented in the introduction are not independent equations. We will show that all of them follow from  \eqref{eq.imag.killing} and a rewritten version of \eqref{eq.Uphi}. 
In particular, we will see that for a generalized imaginary Killing spinor $\phi$ equation \eqref{eq.Uphi} implies  \eqref{char.quad.cond}, unless the vector field~$U_\phi$ vanishes everywhere. 
In the introduction  \eqref{def.dirac.current}--\eqref{def.small.uphi} are a mixture of definitions and relations. Let us rewrite them in a form which is more suitable to clarify their dependences.

We assume that $(N,h)$ is a connected Riemannian spin manifold. 
Let $\Sigma N\to N$ be the associated spinor bundle. Compared to the introduction we slightly simplify our notation: we write $\cdot$ here for the Clifford multiplication instead of writing $\star$ which was used in the introduction in order to distinguish it from other Clifford multiplications.

In the following we have a spinor $\phi$, i.e.\ a smooth section $\phi\in \Gamma(\Sigma N)$, a real-valued smooth function $u\in C^\infty(N)$, a vector field $U\in \Gamma(TN)$ and an endomorphism $W\in \Gamma(TN)$.
Be aware that we simplify again the notation, by writing $W$ for the endomorphism which was called $\ol W$ in the introduction.  The equations \eqref{def.dirac.current}--\eqref{def.small.uphi} turn into

\begin{align}
\kern12mm h(U,X)& = -i \<X\cdot \phi,\phi\>,&&\forall X\in TN,,&\kern12mm\label{def.dirac.current.mod}\\
i\phi(p)&\in \{V\cdot\phi\mid V\in T_pN\},&&\forall p\in N,\label{char.quad.cond.mod}\\
\nabla^N_X\phi& = \frac{i}2 W(X)\cdot \phi,&&\forall X\in TN,\label{eq.imag.killing.mod}\\
  U \cdot\phi & = i u\,\phi,&& \label{eq.Uphi.mod}\\
   u^2 &=  h(U,U).&&\label{def.small.uphi.mod}
\end{align}

We want to discuss the independence of the constraint equations  \eqref{def.dirac.current.mod} to 
\eqref{def.small.uphi.mod}. We start with some elementary lemmata.

\begin{lemma}\label{FirstLemma}
If \eqref{eq.Uphi.mod} is satisfied for $U$, $u$ and $\phi$ defined in some $p\in N$. Then we have (in this $p\in N$):
  $$\phi=0 \text{ or } h(U,U)=u^2$$
\end{lemma}
\begin{proof}
  $$h(U,U)\,\<\phi,\phi\>=  \<-U\cdot U\cdot \phi,\phi\>= \< U\cdot \phi,U\cdot\phi\>= \<iu\,\phi,iu\,\phi\>=u^2\,\<\phi,\phi\>.$$
\end{proof}

\begin{lemma}
Let $p \in N$ and let \eqref{eq.Uphi.mod} be satisfied for $U \in T_pN$, $u \in \R$ and $\phi \in \Sigma_g N_p$. 
We assume that $U_\phi$ is defined by \eqref{def.dirac.current}, i.e.  \eqref{def.dirac.current.mod} holds for $U_\phi$ instead of $U$ at the point $p$. Then $U$ and $U_\phi$ are linearly dependent.
\end{lemma}
In that sense \eqref{eq.Uphi.mod} implies \eqref{def.dirac.current.mod} up to a constant. Obviously for globally defined $U$, $u$ and $\phi$, the proportionality factor does not have to be constant. We obtain $\lambda_1 U=\lambda_2U_\phi$ for some nowhere vanishing function $\lambda:M\to \mR^2$.
\begin{proof}
W.l.o.g.\ $U\neq 0$, $\phi\neq 0$ at $p\in N$. By Lemma \ref{FirstLemma} it follows that $u(p) \neq 0$. We calculate for $X\perp U$:
 \begin{eqnarray*}
   0& =& -2h(X,U)\,\<\phi,\phi\>= \<X\cdot U\cdot \phi,\phi\> + \< U\cdot X\cdot  \phi,\phi\>\\ 
    &=& \<X\cdot U\cdot \phi,\phi\> + \<\phi,  X\cdot U\cdot  \phi\> = 2 \Re\,\<X\cdot U\cdot \phi,\phi\>.
 \end{eqnarray*}
Furthermore
  $$\<X\cdot U\cdot \phi,\phi\>= \<iu\,X\cdot \phi,\phi\>=iu \,\<X\cdot \phi,\phi\>=-u\,h(U_\phi,X).$$
This implies $h(U_\phi,X)=0$.
 \end{proof}

\begin{lemma}[{In \cite[Lemma~5 in Sec.~5.2]{baum.leistner.lecture.notes:HH}}]\label{dirac.curr.prop.equiv}
Assume $U$ and $\phi$ satisfy \eqref{def.dirac.current.mod}. Then \eqref{char.quad.cond.mod} is equivalent to
 $h(U,U) =\|\phi\|^4$.
\end{lemma}
\begin{proof}We prove the statement in each $p\in N$, so we consider $U\in T_pN$ and $\phi\in \Sigma_pN$. W.l.o.g.\ $\phi\neq 0$.  If $(e_j)$ is an orthonormal basis of $T_pN$, then $(\rho_j)$ 
with $\rho_j:=\frac{1}{\|\phi\|}e_j\cdot \phi$ is an orthonormal basis of  $E_{\phi}:=\{V\cdot \phi\mid V\in T_pN\}.$
We write 
  $$i\phi = W\cdot \phi +\psi$$
for some $W\in T_pN$ and $\psi\perp  \{V\cdot \phi\mid V\in T_pN\}$. 
Thus 
  $$\<i\phi,\rho_j\>=\frac{i}{\|\phi\|} \<\phi,e_j\cdot\phi\>=\frac{1}{\|\phi\|}h(U,e_j).$$
We conclude 
  $$h(U,U)=\sum_{j=1}^n h(U,e_j)^2= \|\phi\|^2 \sum_{j=1}^n\<i\phi,\rho_j\>^2=  \|\phi\|^2 \left(\|i\phi\|^2 -\|\psi\|^2\right).$$
This implies that $h(U,U)= \|\phi\|^4$ if and only if $\psi=0$.
\end{proof}

Now let $W\in \End(TN)$, not necessarily symmetric. Recall that our notation is slightly simplified if compared to the introduction: the $\ol W$ of the introduction is $W$ in this appendix.

\begin{proposition}[Dichotomy Proposition]\label{prop.dich}
Let $(N,h)$ be a connected Riemannian manifold and let $W\in \Gamma(\End(TN))$ be a field of endomorphisms. We assume that $U=U_\phi$ for $\phi$ satisfy \eqref{def.dirac.current.mod}, \eqref{eq.imag.killing.mod} and \eqref{eq.Uphi.mod}
for some $u\in C^\infty(N)$.

Then $W^T(U_\phi) =- \grad \langle\phi, \phi\rangle$. If $U_\phi\equiv 0$, then $\|\phi\|$ is constant. If $U_\phi \not\equiv 0$, then $U_\phi$ and $\phi$ vanish nowhere and $u= \|U_\phi\|=\langle\phi, \phi\rangle=\|\phi\|^2$. 
\end{proposition}
Here $W^T$ denotes the endomorphism in $\End(TN)$ adjoint to $W$.

Let us compare this propositition to a similar statement by H.~Baum and Th.~Leistner. In the case that $W$ is symmetric, it yields a criterion implying~\eqref{eq.Uphi.mod}. 

\begin{lemma}[{\cite[Lemma~5 in Sec.~5.2]{baum.leistner.lecture.notes:HH}}]\label{lemma.bl.lectures}
Let $(N,h)$ be a connected Riemannian manifold with a non-zero spinor field $\phi\in \Gamma(\Sigma N)$ and a field of symmetric endomorphisms $W\in \Gamma(\End(TN))$ satisfying \eqref{eq.imag.killing.mod}, 
and let $U=U_\phi$ be defined by \eqref{def.dirac.current.mod}.
Then we have $\nabla U_\phi=-\|\phi\|^2 W$. Furthermore $q_\phi:=\|\phi\|^4-h(U_\phi,U_\phi)$ is non-negative and constant. Moreover, if we define $\psi$ as in the proof of Lemma~\ref{dirac.curr.prop.equiv}, then $q_\phi=\|\phi\|^2 \cdot \|\psi\|^2$. If $q_\phi=0$, then \eqref{eq.Uphi.mod} 
holds for $u:=\|U_\phi\|=\|\phi\|^2$. 
\end{lemma}

Note that in the case of an imaginary Killing spinor, i.e.\ $W=\mu\id$, then $q_\phi$ is related to the constant $Q_\phi$ defined in \cite{baum.friedrich.grunewald.kath:91} for any twistor spinor by $Q_\phi=n^2 \mu^2q_\phi$. Note that imaginary Killing spinors are both twistor spinors and generalized imaginary Killing spinors, but there are generalized imaginary Killing spinors, which are not twistor spinors and vice versa. According to \cite[Chap.~7]{baum.friedrich.grunewald.kath:91} an imaginary Killing spinor is of type~I, if and only if $Q_\phi=0$. Otherwise it is of type~II. Any complete Riemannian manifold carrying a type~II imaginary Killing spinor (with $\mu\neq 0$) is homothetic to the hyperbolic space \cite[Sec.~7.2]{baum.friedrich.grunewald.kath:91}. If it is of type~I, then it arise from a warped product construction as in our Section~\ref{sec.constr.sol}, see \cite[Sec.~7.3]{baum.friedrich.grunewald.kath:91}.

\begin{proof}[Proof of Proposition~\ref{prop.dich}] First we compute 
\begin{eqnarray*}
X\< \phi, \phi\> = \<\nabla_X \phi, \phi\> + \<\phi, \nabla_X \phi\> &=& \frac{i}{2} \left(\left\<W(X) \cdot \phi, \phi\right\> - \left\<\phi, W(X) \cdot \phi\right\>\right)\\
&=& i \<W(X) \cdot \phi, \phi\> = -h( U_\phi,W(X))\\ 
&=& -h(W^T (U_\phi),X).  
\end{eqnarray*}

With $X\< \phi, \phi\> =h(\grad\< \phi, \phi\>,X)$ we obtain  $\grad \<\phi, \phi\> = -W^T (U_\phi)$. Obviously this implies that $\|\phi\|$ is constant in the case $U_\phi\equiv 0$.

We now consider the case  $U_\phi \not\equiv 0$. Note that Equation~\eqref{eq.imag.killing.mod} is equivalent to saying that $\phi$ is a parallel section for the connection $\ol\na_X\phi:=\na_X^N\phi- \frac{i}2 W(X)\cdot \phi$. This implies:
If we have some $p\in M$ with $\phi(p)=0$, then $\phi$ has to vanish on all of $N$, and thus $U_\phi\equiv 0$, the case already solved. So let us assume $\phi(p)\neq 0$ for all $p\in M$.
Note that  \eqref{eq.Uphi.mod} implies 
  $$\|U_\phi\|= \frac{\|U_\phi\cdot\phi\|}{\|\phi\|}=\frac{\|iu\phi\|}{\|\phi\|}=|u|.$$ 

We calculate

\begin{eqnarray*}
u^2 \<\phi, \phi\> = \<iu\, \phi, iu\, \phi\> &=& \<U_\phi \cdot \phi, U_\phi \cdot \phi\>\\
 &=& h(U_\phi,U_\phi) \<\phi, \phi\> \\
 &=& -i \<U_\phi \cdot \phi, \phi\> \<\phi, \phi\>\\
 &=& - i \<iu\, \phi, \phi\> \<\phi, \phi\>\\
 &=& u \<\phi, \phi\>^2 , 
\end{eqnarray*}

i.e.\ at each point $p \in N$ we have $u(p) = 0$ or $u(p)= \<\phi(p), \phi(p)\>=\|\phi(p)\|^2$.
The sets $\{p\in N\mid u(p)=0\}$ and $\{p\in N\mid u(p)=\|\phi(p)\|^2\}$ are closed and disjoint, thus the connectedness of $N$ implies that either $u\equiv 0$ or  $u\equiv \<\phi, \phi\>$. In the case $u\equiv 0$ we obtain 
$U_\phi\equiv 0$, and we are again back in the case already solved.
So we conclude $u\equiv \<\phi, \phi\>=\|\phi\|^2>0$ and thus $u=\|U_\phi\|$.
So everything is proven.
\end{proof}

Now we discuss our main case of interest, i.e.\ that $\phi$ is a
generalized imaginary Killing spinor which is by definition a solution 
of $\nabla_X\phi=\frac{i}2 W(X)\cdot \phi\;\forall X\in TN$ 
with $W\in \End(TN)$ symmetric. We assume $N$ to be connected and $\phi\not\equiv 0$ which implies as we have seen that $\phi$ vanishes nowhere. According to Lemma~\ref{lemma.bl.lectures} we obtain the equation $\nabla U_\phi=-\|\phi\|^2 W$ and the fact that $q_\phi:=\|\phi\|^4-\|U_\phi\|^2$ is a non-negative constant. In the case $q_\phi=0$ (denoted by ``type~I'' in \cite[Sec.~5.2]{baum.leistner.lecture.notes:HH}) we know further that \eqref{eq.Uphi.mod} holds for $u=\|U_\phi\|$.

\begin{corollary}Let $(N,h)$ be a connected Riemannian spin manifold. 
Let $\phi$ be a  generalized imaginary Killing spinor, i.e.\ a solution of \eqref{eq.imag.killing.mod}
for a field of symmetric endomorphisms $W$. Let again $U=U_\phi$ be defined by \eqref{def.dirac.current.mod}. 
We assume $W\not\equiv 0$, thus $\nabla U_\phi\not\equiv 0$ and hence $U_\phi\not\equiv 0$.
Then the following are equivalent:
\begin{enumerate}[{\rm (a)}]
\item\label{point1} $\phi$ is of type~I,
\item\label{point2} $\phi$ satisfies $\|U_\phi\|=\|\phi\|^2$,
\item\label{point3} $\phi$ satisfies equation \eqref{char.quad.cond} (or equivalently \eqref{char.quad.cond.mod}),
\item\label{point4} Condition  \eqref{eq.Uphi.mod} holds for $u:=\|U_\phi\|$.
\end{enumerate}
\end{corollary}

\begin{proof}\ 

\eqref{point1}$\Longleftrightarrow $ \eqref{point2} holds by definition of ``type~I''.

\eqref{point2}$\Longleftrightarrow $ \eqref{point3} is stated in Lemma~\ref{dirac.curr.prop.equiv}.

\eqref{point2}$\Longrightarrow $ \eqref{point4} is part of the statement of Lemma~\ref{lemma.bl.lectures}.

\eqref{point4}$\Longrightarrow $ \eqref{point2} is part of the statement of Proposition~\ref{prop.dich}.
\end{proof}
We have seen, in particular, that for a generalized imaginary Killing spinor~$\phi$ equation
\eqref{eq.Uphi.mod} implies  \eqref{char.quad.cond.mod} for $U=U_\phi$ defined by \eqref{def.dirac.current.mod}, unless the vector field $U_\phi$ and the endomorphism $W$ vanish everywhere.

\section{More on Hypersurfaces}\label{sec.hypersurfaces}

In this appendix we prove Lemma~\ref{lem.hyperflaechenformel}, i.e.\ formula~\eqref{bernd.habil.formel}. 

It is known since long that one cannot restrict spinors to a hypersurface in a way preserving the connection. The difference of the connections depends on the second fundamental form or equivalently the Weingarten map. 
This effect is in some applications very helpful, e.g. in the case of surfaces in Euclidean space, where it leads to the spinorial version of the Weierstrass representation, see \cite{friedrich:98} for a good presentation 
or see \cite{kusner.schmitt:p96} for an earlier, up to branching point aspects complete, but less conceptual publication, based on \cite{kusner.schmitt:p95}. 
How to restrict spinors to hypersurfaces and the effect on the connection was already discussed in mathematical physics in the Riemannian \cite{trautman:92,trautman:95} and Lorentzian \cite{witten:81} context, and in spectral theory \cite{baer:98}.

As different convention are used in the literature and as we follow, similar to  \cite[Prop.~5.3.1]{ammann:habil}, another convention than the well-written exposition \cite{baer.gauduchon.moroianu:05} 
we want to give a detailed proof of Lemma~\ref{lem.hyperflaechenformel} in this appendix.

As in Subsection~\ref{subsec.hyp.surf} and Appendix~\ref{sec.hypersurfaces} we assume that $(N,h)$ is an $n$-dimensional Riemannian spin manifold, $N=M\times (a,b)$, $h=g_s+ds^2$ for a family of metrics $(g_s)$, $s\in (a,b)$ on $M$. Let $\nabla^{M,g_s}$ be the Levi-Civity connection of $(M,g_s)$ and $\nabla^N$ the one of $(N,h)$. We write $\nu:=\partial/\partial s$ for the unit normal vector field of $M\times\{s\}$ in $N$.
For $X\in TM\cong T(M\times\{s\})\subset TN $,  and $Y \in\Gamma(TM)$ we have 
  $$\nabla^N_XY=\nabla^{M,g_s}_XY+\II(X,Y),\qquad  g_s(W(X),Y)\nu=\II(X,Y),$$
and  we have, see e.g.\ \cite[(4.1)]{baer.gauduchon.moroianu:05} $g_s(W(X),Y)= -\frac12 (d/ds)\,g_s$. Note that the second fundamental form $\II$ has values in the normal bundle. Again, we use $\star$ for the Clifford multipication on $N$.

For concrete calculations we choose an open subset $U\subset M\times (a,b)=N$ and for every $(x,s)\in U$ we choose a positively oriented $g_s$-orthonormal basis  $q_M=(e_1,\ldots,e_{n-1})$ of $T_xM$, smoothly depending on $x$ and~$s$. 
Thus $q_M$ is a local section of $P_{\SO, M}(N)$. In other words: 
$q_M$ is a frame of the vertical bundle of $N\to (a,b)$. Furthermore  
$q_N:=(e_1,\ldots,e_n)$, $e_n:=\nu$ is a frame for $(N,h)$, i.e.\ a local section
of $P_{\SO}(N)$. 
We define the associated Christoffel symbols by
  $$\nabla^{M,g_s}_{e_i}e_j=\sum_{k=1}^{n-1} \Gamma^{M\,k}_{ij} e_k,\qquad \nabla^{N}_{e_i}e_j=\sum_{k=1}^{n} \Gamma^{N\,k}_{ij} e_k.$$

Now let $\tilde q_M$ resp.\ $\tilde q_N$ be a spinorial lift of $q_M$ resp.\ $q_N$, i.e.\ a local section of $P_{\Spin,M}(N)$ resp.\ $P_{\Spin}(N)$, such that postcomposing $\tilde q_M$ resp.\ $\tilde q_N$ with $P_{\Spin,M}(N)\to P_{\SO,M}(N)$ resp.\ $P_{\Spin}(N)\to P_{\SO}(N)$ yields $q_M$ resp.\ $q_N$.
On $U$ we can write a spinor $\Phi$, i.e.\ a section of $\Phi\in \Gamma(\Sigma^{(\#)} N)$ as $\Phi=[\tilde q_N,\sigma]$, with $\sigma:U\to \Sigma^{(\#)}_n$.
We also may view as $\Sigma^{(\#)} N$ as an associated bundle to $P_{\SO, M}(N)$, and with proper identifications we get $[\tilde q_N,\sigma]=[\tilde q_M,\sigma]$.
The connection $\nabla^N$ defines the standard Levi-Civita connection on $\Sigma N$, again denoted by  $\nabla^N$. On the other hand, $\nabla^{M,g_s}$ defines a connection on $\Sigma N|_{M\times \{s\}}$.

\begin{proposition}
For any $i=1,\ldots,n-1$ and  $\Phi\in \Gamma(\Sigma^{(\#)} N)$ we have 
  $$\nabla_{e_i}^N \Phi = \nabla_{e_i}^{M,g_s} \Phi + {1\over 2}
    \sum_{j=1}^{n-1} e_j\star\big(\II(e_i,e_j)\big)\star \Phi.$$
\end{proposition}

\begin{proof}
For any $i,j,k\in\{1,\ldots,n-1\}$ we have
  $$\Gamma^{N\,k}_{ij}= \Gamma^{M\,k}_{ij},\qquad\Gamma^{N\,n}_{ij}\,e_n=-\Gamma^{N\,j}_{in}\,e_n = \II(e_i,e_j),\qquad \Gamma^{N\,n}_{in}=0.$$

Let $(E_1,\ldots,E_n)$ denote the canonical basis of $\mR^n$.
Writing the connection in local coordinates we obtain
for $i=1,2,\ldots,n-1$
\begin{eqnarray*}
  \left(\nabla_{e_i}^N[\ti q_M,\sigma]\right)&=&\bigg[\ti q_M,\pa_{e_i}\sigma + 
           {1\over 4}\sum_{j,k=1}^{n} \Gamma^{N\,k}_{ij} E_j\star E_k\star \sigma\bigg]\\
 &=&\underbrace{\bigg[\ti q_M,\pa_{e_i}\sigma +
           {1\over 4}\sum_{j,k=1}^{n-1} \Gamma^{N\,k}_{ij} E_j\star E_k\star \sigma\bigg]}_{\nabla_{e_i}^{M,g_s} [\ti q_M,\sigma]}
+  \underbrace{\bigg[\ti q_M,{1\over 4}\sum_{j=1}^{n-1} \Gamma^{N\,n}_{ij} E_j\star E_n\star \sigma\bigg]}_{ {1\over 4}
    \sum_{j=1}^{n-1} e_j\cdot \big(\II(e_i,e_j)\big)\cdot [\ti q_M,\sigma]}\\
 &&{}+\underbrace{\bigg[\ti q_M,{1\over 4}\sum_{k=1}^{n-1} \Gamma^{N\,k}_{in} E_n\star E_k\star \sigma\bigg]}_{ {1\over 4}
    \sum_{j=1}^{n-1} e_j\star \big(\II(e_i,e_j)\big)\star [\ti q_M,\sigma]} \\
      & = &\nabla_{e_i}^{M,g_s} [\ti q_M,\sigma] + {1\over 2}
    \sum_{j=1}^{n-1} e_j\star \big(\II(e_i,e_j)\big)\star [\ti q_M,\sigma].\qedhere
\end{eqnarray*}
\end{proof}
Obviously we have the following calculation in the Clifford algebra
 $$  \sum_{j=1}^{n-1} e_j\star\II(e_i,e_j)=  \sum_{j=1}^{n-1} e_j\star g_s(W(e_i),e_j)\nu=W(e_i)\star \nu,$$
and thus we obtain for $X\in TM$ and   $\Phi\in \Gamma(\Sigma^{(\#)} N)$ the equation
  $$\nabla_X^N \Phi = \nabla_X^{M,g_s} \Phi + {1\over 2}
    W(X)\star \nu \star \Phi.$$
As the maps $\Joker$ in Lemma~\ref{lem.hyperflaechenformel} are constructed from an algebraic map using the associated bundle construction, they are $\nabla^{M,g_s}$-parallel, i.e.\ 
  $$\nabla_X^{M,g_s}\Joker \phi= \Joker (\nabla_X^{M,g_s} \phi )$$
for all $X\in TM$. Now  Lemma~\ref{lem.hyperflaechenformel} follows immediately
by setting $\Phi:=\Joker(\phi)$.

\section{Change of orientation}\label{sec.change.of.orientation}

In this appendix we will prove the following proposition.

\begin{proposition}\label{prop.change.or}
Let $(M,g)$ be a Riemannian manifold with an orientation~$\Or$, $m=\dim M$.
Let $P_{\Spin}(M,g,\Or)\to  P_{\SO}(M,g,\Or)$ be a spin structure on $(M,g,\Or)$ with associated 
spinor bundle  $\Sigma (M,g,\Or):=P_{\Spin}(M,g,\Or)\times_{\sigma_m}\Sigma_m$ resp. $\Sigma^\# (M,g,\Or):=P_{\Spin}(M,g,\Or)\times_{\sigma_m}\Sigma_m^\#$.\\ 
Then there is a spin structure $P_{\Spin}(M,g,-\Or)\to  P_{\SO}(M,g,-\Or)$  on $(M,g,-\Or)$ with the following properties:\\
Case $m$ even: let  $\Sigma (M,g,-\Or):=P_{\Spin}(M,g,-\Or)\times_{\sigma_m}\Sigma_m$ be the associated spinor bundle, then there is a parallel isometric, complex linear bundle isomorphism $\Psi:\Sigma (M,g,-\Or) \to \Sigma (M,g,\Or)$, commuting with Clifford multiplication. This map $\Psi$ maps $\left(\Sigma (M,g,-\Or)\right)^+$ to $\left(\Sigma (M,g,\Or)\right)^-$ and $\left(\Sigma (M,g,-\Or)\right)^-$ to $\left(\Sigma (M,g,\Or)\right)^+$.\\
Case $m$ odd: let  $\Sigma^{(\#)} (M,g,-\Or):=P_{\Spin}(M,g,-\Or)\times_{\tilde\sigma_m}{\Sigma}_m^{(\#)}$ be one of the associated spinor bundles. Then there is a parallel isometric, complex linear bundle isomorphism $\Psi:\Sigma (M,g,-\Or)\to \Sigma^\# (M,g,\Or)$, commuting with Clifford multiplication. The same statement holds if we exchange the role of $\Sigma_m$ and $\Sigma_m^{\#}$ in the definitions of $\Sigma (M,g,\pm \Or)$.
\end{proposition}

We define a map  $\rho:P_{\SO}(M,g,\Or-)\to  P_{\SO}(M,g,\Or)$ by
$\rho\big((e_1,e_2,\ldots,e_m)\big)\definedas (-e_1,e_2,\ldots,e_m)$ for any $(-\Or)$-oriented 
orthonormal frame $\cE\definedas (e_1,\ldots,e_m)$.

As a covering of smooth manifolds we define 
$P_{\Spin}(M,g,-\Or)\to  P_{\SO}(M,g,-\Or)$  as the pullback of $P_{\Spin}(M,g,\Or)\to  P_{\SO}(M,g,\Or)$ by the diffeomorphism $\rho$. Let $\tilde\rho:P_{\Spin}(M,g,-\Or)\to P_{\Spin}(M,g,\Or) $ be the diffeomorphism
defined by the following pull-back square:
\begin{center}  \begin{tikzpicture}[node distance=2cm, auto]
    \node (A) at (0,0) {$P_{\Spin}(M,g,-\Or)$};
    \node (B) at (4,0) {$P_{\Spin}(M,g,\Or)$};
    \node (C) at (0,-2) {$P_{\SO}(M,g,-\Or)$};
    \node (D) at (4,-2) {$P_{\SO}(M,g,\Or)$};
    \draw[->] (A) to node {$\tilde \rho$} (B);
    \draw[->] (C) to node {$\rho$} (D);
    \draw[->] (A) to node {} (C);
    \draw[->] (B) to node {} (D);
   \end{tikzpicture}
 \end{center}
 
However in order to get an appropriate structure as $\Spin(m)$-principal bundle on 
$P_{\Spin}(M,g,-\Or)$, some care is necessary, as $\rho$ is not $\SO(m)$-equivariant. 
If we define 
$J:=\diag(-1,1,1,1,...1)$ and abbreviate $\cE\definedas (e_1,\ldots,e_m)$, then we have $\rho(\cE)=\cE\cdot J$ and thus
$\rho(\cE A)=\rho(\cE)J^{-1}AJ$.

Conjugation with $J:=J^{-1}$ is a Lie group automorphism of $\SO(m)$ and lifts to  
$\Spin(m) \subset \Cl_m$, as conjugation with $E_1 \definedas
(1,0...,0)$ in the Clifford algebra sense. For any $\witi\cE\in P_{\Spin}(M,g,-\Or)$ and $B\in \Spin(m)$
we define
\begin{equation*}
\witi\cE B
\definedas  
{\tilde\rho\,}^{-1}\Big(\tilde\rho(\witi\cE)\big(E_1B(-E_1)\big)\Big).
\end{equation*}
This definition turns  $P_{\Spin}(M,g,-\Or)$ into a $\Spin(m)$-principal bundle 
and the map 
$P_{\Spin}(M,g,-\Or)\to  P_{\SO}(M,g,-\Or)$ is then $\Spin(m)\to\SO(m)$-equivariant.

Let as before $\si:\Cl_m\to \End(\Si_m^{(\#)})$ be an irreducible representation 
of the Clifford algebra, and let $\Si^{(\#)} (M,g,\Or) \definedas P_{\Spin}(M,g,\Or)\times_\si \Si_m^{(\#)}$ with the obvious modifications for $-\Or$.

\begin{lemma}[Lift to the spinor bundle]\label{lemma.a}
The map 
\begin{equation*}
P_{\Spin}(M,g,-\Or) \times \Si_m^{(\#)} \ni 
(\witi\cE,\phi) \mapsto (\ti\rho(\witi\cE),\si(E_1)\phi)
\in P_{\Spin}(M,g,\Or) \times \Si_m^{(\#)}
\end{equation*} 
is compatible with the equivalence relation given by $\si$.
Thus
it descends to a map 
\begin{equation*}
\rho_\#:\Si^{(\#)} (M,g,-\Or)= P_{\Spin}(M,g,-\Or) \times_\si \Si_m^{(\#)}
\to 
\Si^{(\#)} (M,g,\Or).
\end{equation*}
\end{lemma}

\begin{proof}
 $(\witi\cE B,\si(B^{-1})\phi)$ is mapped to 
\begin{equation*}
\Big(\ti\rho(\witi\cE B),\si(E_1)\si(B^{-1})\phi\Big)
=  
\Big(\ti\rho(\witi\cE)\bigl(E_1B(-E_1)\bigr), \si\bigl((E_1B(-E_1))^{-1}\bigr)\si(E_1)\phi\Big).
\end{equation*} 
This pair is equivalent to $\bigl(\ti\rho(\witi\cE), \si(E_1)\phi\bigr)$ which is the image of  $\bigl(\witi\cE,\phi\bigr)$. 
\end{proof}

\begin{rem}
For future publications it might be helpful here to briefly discuss, what happens if apply this change of orientation twice. Obviously, by replacing $\Or$ by $-\Or$ we get another map 
$\rho_\#:\Si^{(\#)} (M,g,\Or)\to \Si^{(\#)} (M,g,-\Or)$. It is easy to check that we then have 
$(\rho_\#)^2=-\id$.
\end{rem}

In the following sections of an associated vector bundle $V=P\times_\rho W$ --- where~$P$ is a principal $G$-bundle and where~$W$ is a $G$-representation --- 
are written as an equivalence class $[A,w]$ of the pair $(A,w)$ with respect to the action of $\rho:G\to \GL(W)$. 
Here $A$ is a local section of $P$ and $w$ a locally defined function $M\to W$.

\begin{lemma}[Compatibility with the Clifford action]
\begin{equation*}
X\cdot \rho_\#(\phi)
=
-\rho_\#(X\cdot \phi)
\end{equation*}
for $X\in T_pM$, $\phi\in \Si_pM$. 
\end{lemma}

In particular, this lemma implies that although $\rho_\#$ yields an isomorphism between spinor bundles for different orientations, it does not yet have the properties that we request for $\Psi$.
 
\begin{proof}
We view $TM$ as an associated bundle to $P_{\Spin}(M,g,\pm\Or)$ where either sign yields a possible description.
In these descriptions $([\witi\cE,v])$ and $[\ti\rho(\witi\cE), Jv]$ represent the same vector. 
Thus
\begin{equation*}
\begin{split}
  [\witi\cE,v]\cdot \rho_\#([\witi\cE,\phi])
   &= [\ti\rho(\witi\cE), Jv)] \cdot [\ti\rho(\witi\cE), \si(E_1)\phi] \\
   &= [\ti\rho(\witi\cE), \si(Jv)\si(E_1)\phi] \\
   &=[\ti\rho(\witi\cE), -\si(E_1)\si(v)\phi]\\
   &= -\rho_\#([\witi\cE,v]\cdot [\witi\cE,\phi]).
\end{split}
\end{equation*}
Here we used that $(Jv)\cdot E_1 = - E_1 \cdot v$ in $\Cl_m$.
\end{proof}

\begin{lemma}
Let $X\in T_pM$, $\phi\in \Gamma(\Si (M,g,-\Or))$. 
Then
\begin{equation}\label{rho.parallel}
\na_{X}\rho_\#(\phi)=\rho_\#(\na_X\phi).
\end{equation}
\end{lemma}

\begin{proof}
The differential of $\rho:P_{\SO}(M,g,-\Or)\to P_{\SO}(M,g,\Or)$ maps $TP_{\SO}(M,g,-\Or)$ to $TP_{\SO}(M,g,\Or)$. 
The connection $1$-form $\om:\SO(M)\to \so(n)$ then pulls
back according to 
\begin{equation*}
\om((d\rho)(Y)) = J \om(Y) J
\end{equation*}
for $Y\in T_\cE P_{\SO}(M,g,-\Or)$, a lift of $X\in TM$ under the projection
$ P_{\SO}(M,g,-\Or)\to M$. We lift this to a connection $1$-form 
$\ti\om: P_{\Spin}(M,g,-\Or)\to \spin(m)\subset \Cl_m$ which thus transforms as
\begin{equation*}
\ti\om ((d\ti\rho)(\ti Y))
= -E_1 \om(\ti Y) E_1
\end{equation*}
where $\ti Y\in T P_{\Spin}(M,g,-\Or)$ is a lift of $Y$. And this induces the
relation \eqref{rho.parallel}.
\end{proof}
Note that obviously $\rho_\#$ is isometric in each fiber.

Now finally in order to get a map $\Psi$ as in the proposition, we will compose
$\rho_\#$ with a further bundle isomorphism. 
We define $\hat\sigma(X)\definedas -\sigma(X)$ for any $X\in \R^m$. Obviously, $\hat\sigma$ satisfies the Clifford relations, and thus $(\Sigma_m^{(\#)},\hat\sigma)$ describes a representation of $\Cl_m$ which is due to its dimension irreducible. 
The classification of such representations implies that this representation is isomorphic to either $(\Sigma_m,\sigma)$ or (possibly in the case $m$ odd) $(\Sigma_m^\#,\sigma)$. 
If $m$ is even, then we obtain a complex 
vector space isomorphism $K\in \End(\Sigma_m)$ satisfying
\begin{equation}\label{K.antikomm}
K(\sigma(X)\phi)=-\sigma(X)K(X)
\end{equation}
for all $X\in \R^m$ and $\phi\in \Sigma_m$. We can choose $K$ to be isometric.
Similarly, by checking the effect of Clifford multiplication by the volume element we obtain for $m$ odd isometric complex isomorphisms
$K:\Sigma_m\to \Sigma_m^\#$ and $K:\Sigma_m^\#\to \Sigma_m$ satisfying \eqref{K.antikomm} for $X\in\R^m$ and $\phi\in \Sigma_m$ resp.\ $\phi\in \Sigma_m^\#$.
The associated map defined by $K$ defines parallel, fiberwise
isometric complex linear isomorphisms of vector bundles over $\id_M$
 $$K:\Sigma(M,g,\Or)\to \Sigma(M,g,\Or)$$
for $m$ even and 
 $$K:\Sigma(M,g,\Or)\to \Sigma^\#(M,g,\Or),\quad K:\Sigma^\#(M,g,\Or)\to \Sigma(M,g,\Or)$$
for $m$ odd, satisfying \eqref{K.antikomm} for $X\in T_pM$ and $\phi\in \Sigma^{(\#)}(M,g,\Or)$. 

The composition $\Psi:=K\circ \rho_\#$ now satisfies all properties requested in the proposition.

\begin{rem}Let $m$ be even. Recall that a spinor $\phi$ is called positive resp.\ negative
if $\om_\C\phi=\phi$ resp.\ $\om_\C\phi=-\phi$.  
It is easy to check that $\rho_\#$ and $\Psi$  map positive spinors to negative ones and vice versa, while $K$ preserves positivity and negativity of spinors.
\end{rem}

The proof of Proposition~\ref{prop.change.or} is thus complete.

\section{Making paths of metrics divergence free}\label{app.div.free}
In this appendix we show the following well-known lemma.

\begin{lemma}
Let $g_s$, $s\in [0,\ell]$ be a path of Riemannian metrics on a closed manifold $M$. We assume that the dimension of the space of Killing vector fields of $(M,g_s)$ does not depend on $s$.
Then there exists a family of 
diffeomorphisms~$\phi_s:M\to M$, depending smoothly on $s\in [0,\ell]$, $\phi_0=\id_M$, such that $\tilde g_s:=\phi_s^* g_s$ satisfies for all $s\in [0,\ell]$:
\begin{equation}\label{div.equation}
\divv^{\tilde g_s} \frac{d}{ds} \tilde g_s=0.
\end{equation}
\end{lemma}
\begin{proof}We make the following ansatz. 
Let $X_s\in \Gamma(TM)$ be a vector field
smoothly depending on the parameter $s\in [0,\ell]$. Let $\phi_s$ be the flow generated by $X_s$, i.e.\ 
  $$\frac{d}{ds}\phi_s(x)= X_s\big|_{\phi_s(x)}.$$
For this we define $h_s:= \frac{d}{ds} g_s$. Then 
$$\frac{d}{ds} \tilde g_s=\phi_s^* \left(\mathcal{L}_{X_s}g_s + h_s\right),$$
where $\mathcal{L}_{X_s}$ denotes the Lie derivative in the direction of $X_s$.
Thus \eqref{div.equation} is equivalent to 
  $$\divv^{g_s}  \left(\mathcal{L}_{X_s}g_s + h_s\right)=0.$$
Now let $(\divv^{g_s})^*:\Omega^1(M)\to \Gamma(T^*M\odot T^*M)$ be the adjoint of $\divv^{g_s}$. Then, see \cite[Lemma~1.60]{besse:87}, 
  $$\mathcal{L}_{X}g = -2 (\divv^{g})^*(X^b)$$
for all $X\in \Gamma(TM)$ and Riemannian metrics $g$, where $X^b:=g(X,\,\cdot\,)$. We thus see that \eqref{div.equation} is in fact equivalent to 
\begin{equation}\label{div.equation.equiv2}
2\divv^{g_s} (\divv^{g_s})^*\alpha_s= \divv^{g_s} h_s.
\end{equation}
where we set $\alpha_s:=g_s(X_s,\,\cdot\,)$.
By calculation the principal symbol of $P_s:=\divv^{g_s} (\divv^{g_s})^*$ one sees 
that $P_s$ is a self-adjoint elliptic operator, thus has discrete (non-negative) spectrum.
We have $\mathrm{ker}(P_s)=\mathrm{ker}((\divv^{g_s})^*)$. Thus again by   
 \cite[Lemma~1.60]{besse:87} the kernel of $P_s$ is the space of all Killing vector fields of $(M,g_s)$. We furthermore have
 $$\mathrm{im}(P_s)=\mathrm{ker}(P_s)^{\perp}=\mathrm{ker}(\divv^{g_s})^*)^{\perp}=\mathrm{im}(\divv^{g_s}).$$ thus \eqref{div.equation.equiv2} has a unique solution $\alpha_s$ that is $L^2$-orthogonal to any Killing vector field.
  By assumption, the dimension of $\mathrm{ker}(P_s)$ is constant. Thus, the spaces $\mathrm{im}(P_s)$ form a smooth family of isomorphic vector spaces and we have a smooth family of isomorphisms $P_s$ on $\mathrm{im}(P_s)$.
Thus, $\alpha_s$ and hence also $X_s$ and $\phi_s$ depend smoothly on~$s$. This solves the problem.
\end{proof}

Note that we apply the above theorem to a family of Ricci-flat metrics. On closed Ricci-flat Riemannian manifolds every Killing vector field is parallel. Furthermore $X$ is then parallel if and only if $X^b$ is harmonic. Thus the dimension
of the space of Killing vector fields is the first Betti-number and thus independent of $s$.



\end{document}